\documentclass[12pt,reqno]{amsart}
\usepackage{amssymb}
\usepackage{amscd}  
\usepackage{amsxtra}
\usepackage{amsmath}
\usepackage[mathscr]{eucal}
\usepackage{color}
\usepackage[normalem]{ulem} 
\usepackage{hyperref} 
\hypersetup{
colorlinks=true,
allcolors=.,
} 
\usepackage{tikz}

\newfont{\cyr}{wncyr10} 

\theoremstyle{plain}
\newtheorem{theorem}{Theorem}[section]

\newtheorem{lemma}[theorem]{Lemma}
\newtheorem{conjecture}[theorem]{Conjecture}
\newtheorem{proposition}[theorem]{Proposition}
\newtheorem{corollary}[theorem]{Corollary}

\newtheorem*{proposition*}{Proposition}

\theoremstyle{definition}
\newtheorem{remark}[theorem]{Remark}
\newtheorem{definition}[theorem]{Definition}
\newtheorem{example}[theorem]{Example}

\newtheorem{condition}[theorem]{Property}

\theoremstyle{plain}

\numberwithin{equation}{section}

\newcommand{\fA}{{\mathfrak A}}

\newcommand{\fb}{{\mathfrak b}}

\newcommand{\fc}{{\mathfrak c}}

\newcommand{\cF}{{\mathcal F}}

\newcommand{\cJ}{{\mathcal J}}

\newcommand{\1}{{\mathbf 1}}  

\newcommand{\vphi}{{\varphi}}
\newcommand{\vpi}{{\varpi}}

\newcommand{\cL}{{\mathcal L}}

\newcommand{\cN}{{\mathcal N}}

\newcommand{\fp}{{\mathfrak p}}
\renewcommand{\vpi}{{\varpi}}
\newcommand{\bQ}{{\mathbf Q}}
\newcommand{\fq}{{\mathfrak q}}

\newcommand{\br}{{\mathbf r}}

\newcommand{\wt}[1]{{\widetilde{#1}}}

\newcommand{\bZ}{{\mathbf Z}}

\DeclareMathOperator{\Cl}{Cl}
\DeclareMathOperator{\Coker}{Coker}

\DeclareMathOperator{\Det}{Det}

\DeclareMathOperator{\Gal}{Gal}

\DeclareMathOperator{\Hom}{Hom}

\DeclareMathOperator{\Ind}{Ind}

\DeclareMathOperator{\Irr}{Irr}

\DeclareMathOperator{\Ker}{Ker}

\DeclareMathOperator{\Loc}{Loc}
\DeclareMathOperator{\loc}{loc}

\DeclareMathOperator{\Map}{Map}

\DeclareMathOperator{\nr}{nr}
\DeclareMathOperator{\nrd}{nrd}

\DeclareMathOperator{\Ram}{Ram}

\DeclareMathOperator{\sgn}{sgn}

\DeclareMathOperator{\Symp}{Symp}

\newcommand{\mf}[1]{\mathfrak{#1}}

\newcommand{\Z}{{\bf Z}}
\newcommand{\Q}{{\bf Q}}

\newcommand{\iso}{\xrightarrow{\sim}}

\newcommand{\be}{\begin{equation}}
\newcommand{\ee}{\end{equation}}
\newcommand{\lab}[1]{\label{#1}}

\begin{document}
\title[Inverse different]{On the square root of the inverse different} 
\author{A. Agboola, D. Burns, L. Caputo \and Y. Kuang}
\date{Version of October 4, 2022}
\address{A. A.: Department of Mathematics \\
University of California \\ Santa Barbara, CA 93106.}
\email{agboola@math.ucsb.edu} 
\address{D. B.: King's College London\\
Department of Mathematics\\
London WC2R 2LS.}
\email{david.burns@kcl.ac.uk}
\address{L. C.: Madrid, Spain.}
\email{luca.caputo@gmx.com}
\address{Y. K.: Zhuhai, China}
\email{yu.kuang.3@gmail.com}
  
\begin{abstract}
Let $F_{\pi}$ be a finite, Galois-algebra extension of a number field
$F$, with group $G$. Suppose that $F_{\pi}/F$ is weakly ramified, and
that the square root $A_\pi$ of the inverse different
$\mf{D}_{\pi}^{-1}$ is defined. (This latter condition holds if, for
example, $|G|$ is odd.) B. Erez has conjectured that the
class $(A_\pi)$ of $A_\pi$ in the locally free class group $\Cl(\Z G)$
of $\Z G$ is equal to the Cassou-Nogu\`es-Fr\"ohlich root number class
$W(F_{\pi}/F)$ associated to $F_\pi/F$. This conjecture has been
verified in many cases. We establish a precise formula
for $(A_\pi)$ in terms of $W(F_{\pi}/F)$ in all cases that $A_\pi$ is
defined and $F_\pi/F$ is tame, and are thereby able to deduce that, in
general, $(A_\pi)$ is not equal to $W(F_\pi/F)$.
\end{abstract}

\maketitle


\tableofcontents


\section{Introduction} \lab{S:intro}

Let $G$ be a finite group, and suppose that $F_{\pi}/F$ is a
$G$-Galois algebra extension of number fields. Write
$\mf{D}_\pi$ for the different of $F_{\pi}/F$ and
$O_{\pi}$ for the ring of integers of $F_{\pi}$. If $\mf{P}$ is any
prime of $O_{\pi}$, the power $v_{\mf{P}}(\mf{D}_{\pi})$ of
$\mf{P}$ occurring in $\mf{D}_\pi$ is given by
\[
v_{\mf{P}}(\mf{D}_\pi) = \sum_{i = 0}^{\infty} \left(|G^{(i)}_{\mf{P}}| -
1\right),
\]
where $G^{(i)}_{\mf{P}}$ denotes the $i$-th ramification group at
$\mf{P}$ (see \cite[Chapter IV, Proposition 4]{Se79}). This implies
that if, for example, $|G|$ is odd, then the inverse different
$\mf{D}_{\pi}^{-1}$ has a square root, i.e.  there exists a unique
fractional ideal $A_\pi$ of $O_{\pi}$ such that
\[
A_{\pi}^{2} = \mf{D}_{\pi}^{-1}.
\]
(Let us remark at once that if $|G|$ is even, then
$\mf{D}_{\pi}^{-1}$ may well---but of course need not---also
have a square root.)

Recall that $F_{\pi}/F$ is said to be \textit{weakly ramified} if
$G^{(2)}_{\mf{P}} = 0$ for all prime ideals $\mf{P}$ of $O_{\pi}$. B. Erez
has shown that $F_{\pi}/F$ is weakly ramified if and only if $A_\pi$ is a
locally free $O_FG$-module (see \cite[Theorem 1]{E91}). Hence, if
$F_{\pi}/F$ is weakly ramified, it follows that $A_\pi$ is a locally free
$\bZ G$-module, and so defines an element $(A_\pi)$ in the locally
free class group $\Cl(\Z G)$ of $\Z G$. The following result is due to
Erez (see \cite[Theorem 3]{E91}).

\begin{theorem} \label{T:boas}
Suppose that $F_{\pi}/F$ is tamely ramified, and that $|G|$ is odd. Then
$A_\pi$ is a free $\Z G$-module. \qed
\end{theorem}

Based on this and other results, S. Vinatier has made the following
conjecture (cf. \cite[Conjecture]{V} and \cite[Section 1.2]{CV16}): 

\begin{conjecture} \label{C:sv}
Suppose that $F_{\pi}/F$ is weakly ramified, and that $|G|$ is odd. Then
$A_\pi$ is a free $\Z G$-module. \qed
\end{conjecture}

The first detailed study of the Galois structure of $A_\pi$ when
$|G|$ is even is due to the third-named author and Vinatier
\cite{CV16}. By studying the Galois structure of certain torsion
modules first considered by S. Chase \cite{Ch84}, they proved the
following result, and thereby were able to exhibit the first examples
for which $(A_\pi) \neq 0$ in $\Cl(\Z G)$ (see \cite[Theorem
  2]{CV16}).

\begin{theorem} \label{T:cv}
Suppose that $F_{\pi}/F$ is tame and locally abelian (i.e. the decomposition
group at every ramified prime of $F_{\pi}/F$ is abelian). Assume also that
$A_\pi$ exists. Then $(A_\pi) = (O_{\pi})$ in $\Cl(\Z G)$.  \qed
\end{theorem}

A well-known theorem of M. Taylor asserts that, if $F_\pi/F$ is
  tame, then
\be \lab{E:mjt}
(O_\pi) = W(F_\pi/F),
\ee
where $W(F_\pi/F)$ denotes the Cassou-Nogu\`es-Fr\"ohlich root number
class, which is defined in terms of Artin root numbers attached to
non-trivial irreducible symplectic characters of $G$. (In particular,
if $|G|$ is odd, and so has no non-trivial irreducible symplectic
characters, then $W(F_{\pi}/F) = 0$.)


We therefore see that Theorem \ref{T:cv} may be viewed as saying that
if $F_{\pi}/F$ is tame and locally abelian, and if $A_\pi$ exists, then we
have
\[
(A_{\pi}) = (O_{\pi}) = W(F_{\pi}/F).
\]

In light of the results described above, Erez has made the following
(unpublished) conjecture:

\begin{conjecture} \label{C:boas}
Suppose that $F_{\pi}/F$ is weakly ramified, and that $A_\pi$
exists. Then
\[
(A_\pi) = W(F_\pi/F).
\] 
\qed
\end{conjecture}

Conjecture \ref{C:boas} includes Vinatier's Conjecture \ref{C:sv} as a
special case, and was the motivation for the work described in
\cite{CV16}. It also explains almost all previously obtained results
on the $\Z G$-structure of $A_\pi$.
In a different direction, the conjecture is related to recent work of
Bley, Hahn and the second author \cite{BBC} concerning metric
structures arising from $A_\pi$ (for more details of which see the
PhD thesis \cite{K_phd} of the fourth author).
 
In this paper we show that, in general, Conjecture \ref{C:boas} fails
for tame extensions. For each tame extension $F_\pi/F$ we use the
signs at infinity of certain symplectic Galois-Jacobi sums to define
an element $\cJ^*_\infty(F_\pi/F) \in \Cl(\bZ G)$. The class
$\cJ^*_\infty(F_\pi/F)$ is of order at most $2$, and is often, but not
always, equal to zero. We prove the following result.

\begin{theorem} \label{T:tboas}
Suppose that $F_{\pi}/F$ is tame, and that $A_\pi$ exists. Then
\[
(A_\pi) - (O_\pi) = \cJ^*_\infty(F_\pi/F)
\]
i.e. (see \eqref{E:mjt})
\be \lab{E:cdes}
(A_\pi) = W(F_\pi/F) + \cJ^*_\infty(F_\pi/F).
\ee
\qed
\end{theorem}

Our proof of Theorem \ref{T:tboas} combines methods from \cite{AB} and
\cite{AM} involving relative algebraic $K$-theory with the use of
non-abelian Galois-Jacobi sums, the explicit computation by Fr\"ohlich
and Queyrut of the local root numbers of dihedral representations and
a detailed representation-theoretic analysis of the failure (in the
relevant cases) of induction functors to commute with Adams operators.
In particular, it is interesting to compare our use of Galois-Jacobi
sums with the methods of \cite{CV16}, where abelian Jacobi sums play a
critical role.

\begin{remark} \lab{R:ques}
It remains an open question as to whether \eqref{E:cdes} continues to
hold if the tameness hypothesis is relaxed. \qed
\end{remark}

For any integer $m \geq 1$, we write $H_{4m}$ for the generalised
quaternion group of order $4m$. The following result, which is
obtained by combining Theorem \ref{T:tboas} with work of Fr\"ohlich on
root numbers (see \cite{Fr74}), gives infinitely many counterexamples
to Conjecture \ref{C:boas}.

\begin{theorem} \lab{T:ce}
Let $F$ be an imaginary quadratic field such that $\Cl(O_F)$ contains
an element of order $4$. Then for any sufficiently large prime $\ell$
with $\ell \equiv 3 \pmod{4}$, there are infinitely many tame,
$H_{4\ell}$-extensions $F_\pi/F$ such that $A_\pi$ exists and $(A_\pi)
\neq (O_\pi)$ in $\Cl(\bZ H_{4\ell})$.
\end{theorem}

An outline of the contents of this paper is as follows. In Section
\ref{S:rak} we recall certain basic facts about relative algebraic
$K$-theory from \cite{AB} and \cite{AM}. In Section \ref{S:gal}, we
discuss how ideals in Galois algebras give rise to elements in certain
relative $K$-groups. Section \ref{S:loc} contains a description of the
Stickelberger factorisation of certain tame resolvends (see
\cite[Section 7]{AM}) in the case of both rings of integers and square
roots of inverse differents, while Section \ref{S:stick} develops
properties of Stickelberger pairings, and explains how these may be
used to give explicit descriptions of the tame resolvends considered
in the previous section. In Section \ref{S:sum} we recall a number of
facts concerning Galois-Gauss sums. We define Galois-Jacobi sums, and
we establish some of their basic properties. In Section \ref{symp sec}
we compute the signs of local Galois-Jacobi sums at symplectic
characters by combining an analysis of the behaviour of Adams
operators with respect to induction functors together with the theorem
of Fr\"ohlich and Queyrut. In Section \ref{S:fp}, we prove
  Theorem \ref{T:tboas}. Finally, in Section \ref{S:ce}, we prove
  Theorem \ref{T:ce}.
\medskip

\textbf{Acknowledgements:} 
The first-named author learned of the work
of \cite{CV16} and of the conjecture of Erez from conversations with
Philippe Cassou-Nogu\`es and Boas Erez. He is is extremely grateful to
them, as well as to Werner Bley and Cindy Tsang for their subsequent
interest in this project. We are also very grateful to Dominik Bullach
for additional insight into the manner in which counterexamples to
Conjecture \ref{C:boas} can be derived from Theorem \ref{T:tboas} (see
Remark \ref{R:db}).
\medskip

\noindent{}{\bf Notation and conventions.} 

For any field $L$, we write $L^c$ for an algebraic closure of $L$, and
we set $\Omega_L:= \Gal(L^c/L)$. If $L$ is a number field or a
non-archimedean local field (by which we shall always mean a finite
extension of $\bQ_p$ for some prime $p$), then $O_L$ denotes the ring
of integers of $L$. If $L$ is an archimedean local field, then we
adopt the usual convention of setting $O_L = L$.
\smallskip

Throughout this paper, $F$ will denote a number field. For each place
$v$ of $F$, we fix an embedding $F^c \to F_{v}^{c}$, and we view
$\Omega_{F_v}$ as being a subgroup of $\Omega_F$ via this choice of
embedding. We write $I_v$ for the inertia subgroup of $\Omega_{F_v}$
when $v$ is finite.

The symbol $G$ will always denote a finite group upon which $\Omega_F$
acts trivially. If $H$ is any finite group, we write $\Irr(H)$ for the
set of irreducible $F^c$-valued characters of $H$ and $R_H$ for the
corresponding ring of virtual characters. We write $\1_H$ (or simply
$\1$ if there is no danger of confusion) for the trivial character in
$R_H$.

If $L$ is a number field or a local field, and $\Gamma$ is any group
upon which $\Omega_L$ acts continuously, we identify $\Gamma$-torsors
over $L$ (as well as their associated algebras, which are Hopf-Galois
extensions associated to $A_{\Gamma}:= (L^c\Gamma)^{\Omega_{L}}$) with
elements of the set $Z^1(\Omega_L, \Gamma)$ of $\Gamma$-valued
continuous $1$-cocycles of $\Omega_L$ (see \cite[I.5.2]{Se97}).
If $\pi \in Z^1(\Omega_L, \Gamma)$,
then we write $L_\pi/L$ for the corresponding Hopf-Galois extension of
$L$, and $O_\pi$ for the integral closure of $O_L$ in $L_\pi$.  (Thus
$O_{\pi} = L_{\pi}$ if $L$ is an archimedean local field.)  Each such
$L_{\pi}$ is a principal homogeneous space (p.h.s.) of the Hopf
algebra $\Map_{\Omega_L}(\Gamma, L^c)$ of $\Omega_L$-equivariant maps
from $\Gamma$ to $L^c$. It may be shown that if $\pi_1, \pi_2 \in
Z^1(\Omega_L,\Gamma)$, then $L_{\pi_1} \simeq L_{\pi_2}$ if and only
if $\pi_1$ and $\pi_2$ differ by a coboundary. The set of isomorphism
classes of $\Gamma$-torsors over $L$ may be identified with the
pointed cohomology set $H^1(L,\Gamma):=H^1(\Omega_L,\Gamma)$.  We
write $[\pi] \in H^1(L,\Gamma)$ for the class of $L_{\pi}$ in
$H^1(L,\Gamma)$. If $L$ is a number field or a non-archimedean local
field we write $H^1_t(L,\Gamma)$ for the subset of $H^1(L,\Gamma)$
consisting of those $[\pi] \in H^1(L,\Gamma)$ for which $L_{\pi}/L$ is
at most tamely ramified. If $L$ is an archimedean local field, we set
$H^1_t(L,\Gamma) = H^1(L, \Gamma)$. We denote the subset of $H^1_t(L,\Gamma)$
consisting of those $[\pi] \in H^1_t(L,\Gamma)$ for which $L_{\pi}/L$
is unramified at all (including infinite) places of $L$ by
$H^{1}_{nr}(L,\Gamma)$. (So, with this convention, if $L$ is an
archimedean local field, we have $H^{1}_{nr}(L, \Gamma) = 0$.) 
\smallskip

If $A$ is any algebra, we write $Z(A)$ for the centre of $A$. If $A$ is 
an $R$-algebra for some ring $R$, and $R \to R_1$ is an extension of $R$, 
we write $A_{R_1}:= A \otimes_{R} R_1$ to denote extension of scalars from 
$R$ to $R_1$.
\smallskip

\section{Relative algebraic $K$-theory} \lab{S:rak}

The purpose of this section is briefly to recall a number of basic facts
concerning relative algebraic $K$-theory that we shall need. For a
more extensive discussion of these topics, the reader is strongly
encouraged to consult \cite[Section 5]{AM} as well as \cite[Sections 2
  and 3]{AB} and \cite[Chapter 15]{Sw}.

Let $R$ be a Dedekind domain with field of fractions $L$ of
characteristic zero, and suppose that $G$ is a finite group upon which
$\Omega_L$ acts trivially. Let $\fA$ be any finitely-generated
$R$-algebra satisfying $\fA \otimes_R L \simeq LG$.

For any extension $\Lambda$ of $R$, we write $K_0(\fA, \Lambda)$ for
the relative algebraic $K$-group that arises via the extension of
scalars afforded by the map $R \to \Lambda$. Each element of $K_0(\fA,
\Lambda)$ is represented by a triple $[M, N;\xi]$, where $M$ and $N$
are finitely generated, projective $\fA$-modules, and $\xi: M
\otimes_{R} \Lambda \iso N \otimes_R \Lambda$ is an isomorphism of
$\fA \otimes_R \Lambda$-modules.

 Recall that there is a long exact sequence of relative algebraic $K$-theory (see
\cite[Theorem 15.5]{Sw})
\be \lab{E:rkesgen}
K_1(\fA) \xrightarrow{\iota} K_1(\fA \otimes_R \Lambda)
\xrightarrow{\partial^{1}_{\fA,\Lambda}} K_0(\fA, \Lambda)
\xrightarrow{\partial^{0}_{\fA,\Lambda}} K_0(\fA) \to
K_0(\fA \otimes_R \Lambda).
\ee
The first and last arrows in this sequence are induced by the
extension of scalars map $R \to \Lambda$, while the map
$\partial^{0}_{\fA, \Lambda}$ sends the triple $[M, N;\xi]$ to the
element $[M] - [N] \in K_0(\fA)$. 

The map $\partial^{1}_{\fA, \Lambda}$ is defined as follows. The group
$K_1(\fA \otimes_R \Lambda)$ is generated by elements of the form
$(V,\phi)$, where $V$ is a finitely generated, free $\fA \otimes_R
\Lambda$-module, and $\phi: V \xrightarrow{\sim} V$ is an $\fA
\otimes_R \Lambda$-isomorphism. To define $\partial^{1}_{\fA,
  \Lambda}((V,\phi))$, we choose any projective $\fA$-submodule $T$ of
$V$ such that $T \otimes_{\fA} \Lambda = V$, and we set
\[
\partial^{1}_{\fA, \Lambda}((V,\phi)) := [T,T;\phi].
\]
It may be shown that this definition is independent of the choice of
$T$.

Let $\Cl(\fA)$ denote the locally free class group of $\fA$. If
$\Lambda$ is a field (as will in fact always be the case in this
paper), then \eqref{E:rkesgen} yields an exact sequence
\be \lab{E:rkes}
K_1(\fA) \xrightarrow{\iota} K_1(\fA \otimes_R \Lambda)
\xrightarrow{\partial^{1}_{\fA,\Lambda}} K_0(\fA, \Lambda)
\xrightarrow{\partial^{0}_{\fA,\Lambda}} \Cl(\fA) \to 0,
\ee
and this is the form of the long exact sequence of relative algebraic
$K$-theory that we shall use in this paper.

We shall make heavy use of the fact that computations in relative
$K$-groups and in locally free class groups may be carried out using
functions on the characters of $G$. Suppose that $L$ is either a
number field or a local field, and write $R_G$ for the ring of virtual
characters of $G$. The group $\Omega_L$ acts on $R_G$ via the rule
given by
\[
\chi^{\omega}(g) = \omega(\chi(g)),
\]
where $\omega \in \Omega_L$, $\chi \in \Irr(G)$, and $g \in G$. For
each element $a \in (L^cG)^{\times}$, we define $\Det(a) \in \Hom(R_G,
(L^{c})^{\times})$ as follows. If $T$ is any representation of $G$
with character $\phi$, then we set $\Det(a)(\phi):= \det(T(a))$. It
may be shown that this definition is independent of the choice of
representation $T$, and so depends only upon the character $\phi$.

The map $\Det$ is essentially the same as the reduced norm map 
\be \lab{E:nrd}
\nrd: (L^{c}G)^{\times} \to Z(L^cG)^{\times}
\ee
(see \cite[Remark 4.2]{AM}): \eqref{E:nrd} induces an isomorphism
\be \lab{E:4.3}
\nrd: K_1(L^cG) \iso Z(L^cG)^{\times} \simeq \Hom(R_G,
(L^{c})^{\times}),
\ee
and we have $\Det(a)(\phi) = \nrd(a)(\phi)$.

Suppose now that we are working over a number field $F$ (i.e. $L = F$
above). We define the group of finite ideles $J_f(K_1(FG))$ to be the
restricted direct product over all finite places $v$ of $F$ of the
groups $\Det(F_vG)^{\times} \simeq K_1(F_vG)$ with respect to the
subgroups $\Det(O_{F_v}G)^{\times}$. (We shall require no use of the
infinite places of $F$ in the idelic descriptions given below. See
e.g.  \cite[pages 226--228]{CR2} for details concerning this point.)

For each finite place $v$ of $F$, we write
\[
\loc_v: \Det(FG)^{\times} \to \Det(F_vG)^{\times}  \subseteq
\Hom_{\Omega_{F_v}}(R_G, (F_{v}^c)^{\times})
\]
for the obvious localisation map.

Let $E$ be any extension of $F$. Then the homomorphism
\[
\Det(FG)^{\times} \to J_{f}(K_1(FG)) \times \Det(EG)^{\times};\quad x \mapsto
((\loc_v(x))_v, x^{-1})
\]
induces a homomorphism
\[
\Delta_{\fA,E}: \Det(FG)^{\times} \to
\frac{J_{f}(K_1(FG))}{\prod_{v \nmid \infty} \Det(\fA_v)^{\times}} \times
 \Det(EG)^{\times}.
\]

\begin{theorem}\label{T:kdes}
(a) There is a natural isomorphism
\[
\Cl(\fA) \xrightarrow{\sim} \frac{J_{f}(K_1(FG))}{\Det(FG)^{\times}
    \prod_{v \nmid \infty} \Det(\fA_v)^{\times}}.
\]

(b) There is a natural isomorphism
\[
h_{\fA,E}: K_0(\fA, E) \xrightarrow{\sim} \Coker(\Delta_{\fA,E}).
\]

(c) Let $v$ be a finite place of $F$, and suppose that $L_v$ is any
extension of $F_v$. Then there are isomorphisms
\[
K_0(\fA_v, L_v) \simeq K_1(L_vG)/\iota(K_1(\fA_v)) \simeq
 \Det(L_vG)^{\times}/ \Det(\fA_v)^{\times}. 
\]
\end{theorem}

\begin{proof}
Part (a) is due to A. Fr\"ohlich (see e.g \cite[Chapter I]{Fr1}). Part
(b) is proved in \cite[Theorem 3.5]{AB}, and a proof of part (c) is
given in \cite[Lemma 5.7]{AM}.
\end{proof}

%

\begin{remark} \lab{R:comp}
Suppose that $x \in K_0(\fA, E)$ is represented by the idele
$[(x_{v})_{v}, x_{\infty}] \in J_{f}(K_1(FG)) \times \Det(EG)^{\times}$. 
Then $\partial^0(x) \in
\Cl(\fA)$ is represented by the idele $(x_{v})_{v} \in J_{f}(K_1(FG))$.
\qed
\end{remark}

\begin{remark} \lab{R:comp2}
Suppose that $[M,N;\xi] \in K_0(O_FG, E)$, and that $M$ and $N$ are
locally free $\fA$-modules of rank one. An explicit representative in
$J_{f}(K_1(FG)) \times \Det(EG)^{\times}$ of $h_{\fA,E}([M,N;\xi])$ may be
constructed as follows.

For each finite place $v$ of $F$, fix $\fA_v$-bases $m_v$ of $M_v$ and $n_v$
of $N_v$. Fix also an $FG$-basis $n_{\infty}$ of $N_F$, and choose an
isomorphism $\theta: M_F \xrightarrow{\sim} N_F$ of $FG$-modules.

The element $\theta^{-1}(n_{\infty})$ is an $FG$-basis of
$M_F$. Hence, for each place $v$, we may write
\begin{align*}
m_v &= \mu_v \cdot \theta^{-1}(n_{\infty}), \\
n_v &= \nu_v \cdot n_{\infty},
\end{align*}
where $\mu_v, \nu_v \in (F_vG)^{\times}$.

If we write $\theta_E: M_E \xrightarrow{\sim} N_E$ for the isomorphism
afforded by $\theta$ via extension of scalars, then we see that the
isomorphism $\xi \circ \theta^{-1}_{E}: N_E \xrightarrow{\sim} N_E$ is given
by $n_{\infty} \mapsto \nu_{\infty} \cdot n_{\infty}$ for some
$\nu_{\infty} \in (EG)^{\times}$. 

A representative of $h_{\fA,E}([M,N;\xi])$ is given by the image of
$[(\mu_v \cdot \nu^{-1}_{v})_{v}, \nu_{\infty}]$ in $J_{f}(K_1(FG)) \times
\Det(EG)^{\times}$.
\qed
\end{remark}

\begin{remark} \lab{R:kof}
We see from Theorem \ref{T:kdes}(b) and (c) that there are isomorphisms
\[
K_0(\fA, F) \simeq \frac{J_{f}(K_1(FG))}{\prod_{v \nmid \infty} \Det(\fA_{v})^{\times}} \simeq
\frac{\Hom_{\Omega_F}(R_G, J_{f}(F^c))}{\prod_{v \nmid \infty} \Det(\fA_{v})^{\times}}
\simeq \oplus_{v \nmid \infty} K_{0}(\fA_{v}, F_{v}).
\]
There is a natural injection
\begin{align*}
K_0(\fA, F) &\to K_0(\fA, F^c) \\
[M,N;\xi] &\to [M,N;\xi_{F^{c}}],
\end{align*}
where $\xi_{F^c}: M_{F^c} \xrightarrow{\sim} N_{F^c}$ is the isomorphism
obtained from $\xi:M_F \xrightarrow{\sim} N_{F}$ via extension of
scalars from $F$ to $F^c$. It is not hard to check that this map is
induced by the inclusion map
\begin{align*}
J_{f}(K_1(FG)) &\to J_{f}(K_1(FG)) \times \Det(F^{c}G)^{\times} \\
(x_{v})_{v} &\to [(x_{v})_{v}, 1].
\end{align*}
\qed
\end{remark}

We now recall the description of the restriction of scalars map
on relative $K$-groups and locally free class groups in terms of the
isomorphism given by Theorem \ref{T:kdes}(b).

Suppose that $\cF/F$ is a finite extension, and that $E$ is an
extension of $\cF$. Then restriction of scalars from $O_{\cF}$ to
$O_F$ yields homomorphisms
\[
K_0(\fA_{O_{\cF}}, E) \to K_0(\fA, E)
\]
and
\[
\Cl(\fA_{O_{\cF}}) \to \Cl(\fA)
\]
which may be described as follows (see e.g.
\cite[Chapter IV]{Fr1} or \cite[Chapter 1]{Ty84}).

Let $\{\omega\}$ be any transversal of $\Omega_{\cF} \backslash
\Omega_F$. Then the map
\begin{align*}
J_{f}(K_1(\cF G)) \times \Det(EG)^{\times} &\to J_{f}(K_1(FG)) \times
\Det(EG)^{\times}\\
[(y_v)_v, y_{\infty}] &\mapsto \prod_{\omega}
    [(y_v)_v, y_{\infty}]^{\omega}
\end{align*}
induces homomorphisms
\be \lab{E:normk}
\cN_{\cF/F}: K_0(\fA_{O_{\cF}}, E) \to K_0(\fA, E)
\ee
and
\be \lab{E:normc}
\cN_{\cF/F}: \Cl(\fA_{O_{\cF}}) \to \Cl(\fA).
\ee
These homomorphisms are independent of the choice of $\{\omega\}$ and
are equal to the natural maps on relative $K$-groups (resp. locally
free class groups) afforded by restriction of scalars from $O_{\cF}$ to
$O_F$.

We conclude this section by recalling the definitions of
  certain induction maps on relative algebraic $K$-groups and on
  locally free class groups of group rings (see e.g. \cite[Chapter
    II]{Fr1} or \cite[Chapter I]{Ty84}).

Suppose that $G$ is a finite group, and that $H$ is a subgroup of $G$. Let
$E$ be an algebraic extension of $F$. Then extension of scalars from
$O_FH$ to $O_FG$ yields natural homomorphisms
\be \lab{E:indeq1} 
\Ind^G_H: K_0(O_FH, E) \to K_0(O_FG, E)
\ee
and
\be \lab{E:indeq2}
\Ind^G_H:\Cl(O_FH) \to \Cl(O_FG).
\ee
It may be shown that these homomorphisms are induced (via the
isomorphisms described in Theorem \ref{T:kdes}) by the maps
\begin{align*}
&\Ind^G_H: \Hom(R_H, J(F^c)) \to \Hom(R_G, J(F^c)); \\
&\Ind^G_H:\Hom(R_H, (F^{c})^{\times}) \to \Hom(R_G, (F^{c})^{\times})
\end{align*}
given by
\be \lab{E:indeq3}
(\Ind^G_Hf)(\chi) = f(\chi \mid_H), \quad \chi \in R_G.
\ee

It is not hard to check from the definitions that
the following diagram commutes: 
\be \lab{E:indcd} 
\begin{CD}
K_0(O_F H, E) @>{\Ind^G_H}>> K_0(O_F G, E) \\
@V{\partial^0}VV                          @V{\partial^{0}}VV \\
\Cl(O_F H) @>{\Ind^G_H}>> \Cl(O_F G). \\
\end{CD}
\ee


\section{Galois algebras and ideals} \lab{S:gal}

Let $L$ be either a number field or a local field, and suppose that
$\pi \in Z^1(\Omega_L, G)$ is a continuous $G$-valued $\Omega_L$
$1$-cocycle. We may define an associated $G$-Galois $L$-algebra
$L_{\pi}$ by 
\[
L_{\pi} := \Map _{\Omega_L}(^{\pi}G, L^c),
\]
where $^{\pi}G$ denotes the set $G$ endowed with an action of
$\Omega_L$ via the cocycle $\pi$ (i.e. $g^{\omega} = \pi(\omega) \cdot
g$ for $g \in {^{\pi}G}$ and $\omega \in \Omega_L$), and $L_{\pi}$ is
the algebra of $L^c$-valued functions on $^{\pi}G$ that are fixed
under the action of $\Omega_L$. The group $G$ acts on $L_{\pi}$ via
the rule
\[
a^g(h) = a(h \cdot g)
\]
for all $g \in G$ and $h \in {^{\pi}G}$. 

The Wedderburn decomposition of the algebra $L_{\pi}$ may be described
as follows. Set
\[
L^{\pi} := (L^{c})^{\Ker(\pi)},
\]
so $\Gal(L^{\pi}/L) \simeq \pi(\Omega_L)$. Then 
\be \lab{E:wiso}
L_{\pi} \simeq \prod_{\pi(\Omega_L) \backslash G} L^{\pi},
\ee
and this isomorphism depends only upon the choice of a transversal of
$\pi(\Omega_L)$ in $G$. It may be shown that every
$G$-Galois $L$-algebra is of the form $L_{\pi}$ for some $\pi$, and
that $L_{\pi}$ is determined up to isomorphism by the class $[\pi]$ of
$\pi$ in the pointed cohomology set $H^1(L, G)$. In particular, every
Galois algebra may be viewed as being a sub-algebra of the
$L^c$-algebra $\Map(G, L^c)$.

\begin{definition}
The \textit{resolvend} map $\br_G$ on $\Map(G, L^c)$ is defined by
\begin{align*}
\br_G: \Map(G, L^c) &\rightarrow L^cG \\
a &\mapsto \sum_{g \in G} a(g) \cdot g^{-1}.
\end{align*}
(This is an isomorphism of $L^cG$-modules, but it is not an
isomorphism of $L^c$-algebras because it does not preserve
multiplication.)
\qed
\end{definition}

Suppose now that $L_{\pi}/L$ is a $G$-extension, and that $\cL
\subseteq L_{\pi}$ is a non-zero projective $O_LG$-module. Then there
are isomorphisms 
\[
\Map(G, L^c) \simeq \cL \otimes_{O_L} L^c, \quad L^cG \simeq O_LG
\otimes_{O_L} L^c,
\]
and so the triple $[\cL, O_LG; \br_G]$ yields an element of $K_0(O_LG,
L^c)$.

\begin{proposition} \lab{P:des}
Let $F_{\pi}/F$ be a $G$-extension of a number field $F$, and suppose
that $\cL_i \subseteq F_{\pi}$ ($i =1,2$) are non-zero projective
$O_FG$-modules. For each place $v$ of $F$, choose a basis $l_{i,v}$ of
$\cL_{i,v}$ over $O_{F_v}G$, as well as a basis $l_{\infty}$ of
$F_{\pi}$ over $FG$.

\begin{itemize}
\item[(a)] The element $[\cL_i, O_FG; \br_G] \in K_0(O_FG, F^c)$ is represented
by the image of the idele $[(\br_G(l_{i,v}) \cdot
  \br_G(l_{\infty})^{-1})_v, \br_G(l_{\infty})^{-1}] \in J_{f}(K_1(FG))
\times \Det(F^cG)^{\times}$.

\item[(b)] The element
\[
[\cL_1, O_FG; \br_G] - [\cL_2, O_FG; \br_G] \in K_0(O_FG, F^c)
\]
is represented by the image of the idele 
\[
[(\br_G(l_{1,v}) \cdot
  \br_G(l_{2,v}^{-1}))_v, 1] \in J_{f}(K_1(FG)) \times
\Det(F^cG)^{\times}.
\]

\item[(c)] We have that
\[
[\cL_1, O_FG; \br_G] - [\cL_2, O_FG; \br_G] \in K_0(O_FG, F) \subseteq
K_0(O_FG, F^c).
\] 
\end{itemize}
\end{proposition}

\begin{proof}
For each finite place $v$ of $F$, write
\[
l_{i,v} = x_{i,v} \cdot l_{\infty},
\]
with $x_{i,v} \in (F_vG)^{\times}$. Then it follows from Remark
\ref{R:comp2} that $[\cL_i, O_FG; \br_G] \in K_0(O_FG, F^c)$ is
represented by the image of the idele $[(x_{i,v})_{v},
  \br_G(l_{\infty})^{-1}] \in J_{f}(K_1(FG)) \times
\Det(F^cG)^{\times}$. However
\[
x_{i,v} = \br_G(l_{i,v}) \cdot \br_G(l_{\infty})^{-1}
\]
(because the resolvend map is an isomorphism of $F^cG$ and
$F^{c}_{v}G$-modules), and this implies (a). Part (b) now follows
directly from (a).

To show part (c), we first recall that 
\[
K_0(O_FG, F) \simeq \oplus_{v \nmid \infty} K_0(O_{F_v}G, F_v) \simeq
\oplus_{v \nmid \infty} \Det(F_vG)^{\times}/\Det(O_{F_v}G)^{\times},
\]
and that an element $c \in K_0(O_FG, F^c)$ lies in $K_0(O_FG, F)$ if
it has an idelic representative lying in $J_{f}(K_1(FG)) \times
\Det(FG)^{\times} \subseteq J_{f}(K_1(FG)) \times
\Det(F^cG)^{\times}$ (see Remark \ref{R:kof}).

Now a standard property of resolvends
implies that 
\[
\br_G(l_{i,v})^{\omega} = \br_G(l_{i,v}) \cdot \pi(\omega)
\]
for every $\omega \in \Omega_{F_v}$ (see e.g. \cite[2.2]{AM}), and so
we see that $(\br_G(l_{1,v}) \cdot \br_G(l_{2,v}^{-1}))_v \in
(F_vG)^{\times}$ for each $v$. (In fact, as we may take $l_{1,v} =
l_{2,v}$ for almost all $v$, we may suppose that $(\br_G(l_{1,v})
\cdot \br_G(l_{2,v}^{-1}))_v =1$ for almost all $v$.) Hence it now
follows from (b) that $[\cL_1, O_FG; F^c] - [\cL_2, O_FG; F^c] \in
K_0(O_FG, F) $, as claimed.
\end{proof}

It is a classical result, due to E. Noether, that a $G$-extension
$F_{\pi}/F$ is tamely ramified if and only if $O_{\pi}$ is a locally
free (and therefore projective) $O_FG$-module. S. Ullom has shown that
if $F_{\pi}/F$ is tame, then in fact all $G$-stable ideals in
$O_{\pi}$ are locally free. He also showed that if any $G$-stable
ideal $B$, say, in a $G$-extension $F_{\pi}/F$ is locally free, then
all second ramification groups at primes dividing $B$ are equal to
zero (see \cite{U1969}). If $F_{\pi}/F$ is any $G$-extension for which
$|G|$ is odd (and so the square root $A_\pi$ of the inverse different
automatically exists), then Erez has shown that $A_\pi$ is a locally
free $O_FG$-module if and only if all second ramification groups
associated to $F_{\pi}/F$ vanish, i.e. if and only if $F_{\pi}/F$ is
weakly ramified.  In fact, as pointed out by the
  third-named author and Vinatier, \cite[pp. 109, footnote 1]{CV16}
  the proof of \cite[Theorem 1]{E91} shows that if $F_{\pi}/F$ is any
  weakly ramified extension such that $A_{\pi}$ exists, then $A_{\pi}$
  is locally free.


\begin{definition} \lab{D:fc}
Suppose that $[\pi] \in H^{1}_{t}(F,G)$, and that $A_\pi$ exists. We
define
\[
\fc = \fc(\pi) := [A_\pi, O_FG;\br_G] - [O_{\pi}, O_FG;\br_G] \in
K_0(O_FG, F) \subseteq K_0(O_FG, F^c).
\]
\qed
\end{definition}


\section{Local decomposition of tame resolvends} \lab{S:loc}

Our goal in this section is to recall certain facts from \cite[Section
  7]{AM} concerning Stickelberger factorisations of resolvends of
normal integral basis generators of tame local extensions, and to
describe similar results concerning resolvends of basis generators of
the square root of the inverse different (when this exists). 

Let $L$ be a local field, and fix a uniformiser $\vpi = \vpi_L$ of
$L$. Set $q:= |O_L/\vpi_L O_L|$.

Fix also a compatible set of roots of unity $\{ \zeta_m \}$, and
a compatible set $\{ \vpi^{1/m} \}$ of roots of $\vpi$. (Hence if $m$ and
$n$ are any two positive integers, then we have $(\zeta_{mn})^m =
\zeta_n$, and $(\vpi^{1/mn})^{m} = \vpi^{1/n}$.)

Let $L^{nr}$ (respectively $L^{t}$) denote the
maximal unramified (respectively tamely ramified) extension of
$L$. Then
\[
L^{nr}=  \bigcup_{\stackrel{m \geq 1}{(m,q)=1}} L(\zeta_m),\quad
L^t = \bigcup_{\stackrel{m \geq 1}{(m,q)=1}} L(\zeta_m, \vpi^{1/m}).
\]
The group $\Omega^{nr}:=
\Gal(L^{nr}/L)$ is topologically generated by a Frobenius
element $\phi$ which may be chosen to satisfy
\[
\phi(\zeta_m) = \zeta_{m}^{q}, \qquad \phi(\vpi^{1/m}) = \vpi^{1/m}
\]
for each integer $m$ coprime to $q$. Our choice of compatible roots
of unity also uniquely specifies a topological generator $\sigma$ of
$\Omega^r := \Gal(L^t/L^{nr})$ by the conditions 
\[
\sigma(\vpi^{1/m}) = \zeta_m \cdot \vpi^{1/m}, \qquad
\sigma(\zeta_m) = \zeta_m
\]
for all integers $m$ coprime to $q$. The group
$\Omega^{t}:=\Gal(L^{t}/L)$ is topologically generated by
$\phi$ and $\sigma$, subject to the relation
\begin{equation} \label{E:tamerel}
\phi \cdot \sigma \cdot \phi^{-1} = \sigma^{q}.
\end{equation}


The reader may find it helpful to keep in mind the following explicit
example, due to C. Tsang (cf. \cite[Proposition 4.2.2]{T16}), while
reading the next two sections.

\begin{example} \lab{E:ke}
(C. Tsang) Suppose that $L$ contains the $e$-th roots of unity with $(e,q)=1$, and
  set $M := L(\vpi_{L}^{1/e})$. Write $\vpi_M := \vpi_{L}^{1/e}$; then
  $\vpi_M$ is a uniformiser of $M$. Set $H := \Gal(M/L) = \langle s
  \rangle$, say.

Let $n$ be an integer with $0 \leq |n| \leq e-1$, and let us consider the
ideal
\[
\vpi^{n}_{M} O_M = \vpi_{L}^{n/e}O_M
\]
as an $O_LH$-module. Set
\[
\alpha = \frac{1}{e}\sum_{i=0}^{e-1} \vpi_{M}^{n+i} =
\frac{1}{e}\sum_{i=0}^{e-1} \vpi_{L}^{(n+i)/e}.
\]
We wish to explain why
\[
O_LH \cdot \alpha = \vpi_{M}^{n} \cdot O_{M},
\]
and to give some motivation for the definition of the Stickelberger
pairings in Definition \ref{D:sticks} below. 

Suppose that $s(\vpi_M) = \zeta \cdot \vpi_M$, where $\zeta$ is a
primitive $e$-th root of unity. Then for each $0 \leq j \leq e-1$, we
have
\[
s^j(\alpha) = \frac{1}{e}\sum_{i=0}^{e-1}\zeta^{(i+n)j} \vpi_M^{i+n}.
\]
Multiplying both sides of this last equality by $\zeta^{-(l+n)j}$,
where $0 \leq l \leq e-1$ gives
\[
s^j(\alpha) \zeta^{-(l+n)j} = \frac{1}{e} \sum^{e-1}_{i=0} \zeta^{(i-l)j}
\vpi_{M}^{i+n}.
\]
Now sum over $j$ to obtain
\[
\sum_{j=0}^{e-1} s^j(\alpha) \zeta^{-(l+n)j}  = \frac{1}{e}
\sum_{i=0}^{n} \vpi_{M}^{i+n} \sum_{j=0}^{e-1} \zeta^{(i-l)j} =
\vpi_{M}^{l+n}.
\]

So, if for any $\chi \in \Irr(H)$, we choose the unique integer
$(\chi,s)_{H,n}$ in the set 
\[
\{l+n \mid 0 \leq l \leq e-1 \}
\]
such that $\chi(s) = \zeta^{(\chi,s)_{H,n}}$, then we see that
\be \lab{E:bexp}
\Det(\br_H(\alpha))(\chi) = \sum_{j=0}^{e-1} s^j(\alpha) \zeta^{-(l+n)j}  =
\vpi_{M}^{(\chi,s)_{H,n}}.
\ee

The cases $n =0$ and $n = (1-e)/2$ (for $e$ odd) correspond to the
ring of integers and the square root of the inverse different
respectively, and we see the appearance of the relevant Stickelberger
pairing (see Definition \ref{D:sticks} below) in each case.  

It follows from \eqref{E:bexp} that
\[
B_n := \{ \vpi_{M}^{l+n} : 0 \leq l \leq e-1 \} \subseteq O_{L}H \cdot \alpha.
\]
As $B_n$ is an $O_L$-basis of the ideal $\vpi_{M}^{n} \cdot O_{M}$,
and as $\zeta_{e} \in O_L$, we see that 
\[
O_LH \cdot \alpha = \vpi^{n}_{M} \cdot O_{M},
\]
i.e. $\alpha$ is a free generator of $\vpi^{n}_{M} \cdot O_M$ as an $O_{L}H$-module.
\qed
\end{example}

\begin{definition} \label{D:ramphi}
If $g \in G$, we set
\[
\beta_g := \frac{1}{|g|} \sum_{i=0}^{|g|-1} \vpi^{i/|g|};
\]
note that $\beta_g$ depends only upon $|g|$, and so in particular we
have 
\[
\beta_{g} = \beta_{\gamma^{-1}g\gamma}
\]
for every $\gamma \in G$.  We define $\vphi_{g} \in \Map(G, L^{c})$
by setting
\[
\vphi_{g}(\gamma) = 
\begin{cases}
\sigma^{i}(\beta_g)   &\text{if $\gamma = g^i$;} \\
0     &\text{if $\gamma \notin \langle g \rangle$.}
\end{cases}
\]
Then
\begin{equation} \label{E:locres}
\br_G(\vphi_{g}) = \sum_{i=0}^{|g|-1} \vphi_{g}(g^i) g^{-i} =
\sum_{i=0}^{|g|-1} \sigma^{i}(\beta_g) g^{-i}.
\end{equation}
\qed
\end{definition}

Suppose now that $\pi \in Z^1(\Omega_L, G)$, with $[\pi] \in
H^1_t(L,G)$. Write $s:= \pi(\sigma)$ and $t:= \pi(\phi)$. 
We define, $\pi_{r}, \pi_{nr} \in
\Map(\Omega^t, G)$ by setting
\begin{align}
&\pi_r(\sigma^m \phi^n) = \pi(\sigma^m) = s^m , \label{E:rmap} \\
&\pi_{nr}(\sigma^m \phi^n) = \pi(\phi^n) = t^n , \label{E:nrmap}
\end{align}
so that
\[
\pi = \pi_{r} \cdot \pi_{nr}.
\]
It may be shown that in fact $\pi_{nr} \in \Hom(\Omega^{nr}, G)$, and
so corresponds to a unramified $G$-extension $L_{\pi_{nr}}$ of $L$. It
may also be shown that $\pi_{r} \in \Hom(\Omega^r, G)$, corresponding
to a totally (tamely) ramified extension $L^{nr}_{\pi_{r}}/L^{nr}$. If we write
$[\wt{\pi}]$ for the image of $[\pi]$ under the natural restriction
map $H^1(L, G) \to H^1(L^{nr},G)$, then $[\wt{\pi}] = [\pi_{r}]$. The
element $\vphi_s$ is a normal integral basis generator of the
extension $L^{nr}_{\pi_{r}}/L^{nr}$. (See \cite[Section 7]{AM} for
proofs of these assertions.) If in addition $|s|$ is odd, then the
inverse different of $L_{\pi}/L$ has a square root $A_\pi$, and
\[
A_\pi = \vpi^{(1-|s|)/2|s|} \cdot O_{\pi}.
\]

We can now state the Stickelberger factorisation theorem for
resolvends of normal integral bases.

\begin{theorem} \label{T:stickfac}
If $a_{nr} \in L_{\pi_{nr}}$ is any normal integral basis generator of
$L_{\pi_{nr}}/L$, then the element $a \in L_{\pi}$ defined by
\begin{equation} \label{E:stickfac1}
\br_G(a_{nr}) \cdot \br_G(\vphi_s) = \br_G(a)
\end{equation}
is a normal integral basis generator of $L_{\pi}/L$.
\end{theorem}

\begin{proof}
See \cite[Theorem 7.9]{AM}.
\end{proof}

We shall now describe a similar result (due to C. Tsang when $G$ is
abelian) concerning $O_LG$-generators of the square root of the
inverse different of a tame extension of $L$.

\begin{definition} \lab{D:tphi}
Suppose that $g \in G$ and that $|g|$ is odd. Set
\[
\beta^{*}_{g} = \frac{1}{|g|} \sum_{i=0}^{|g|-1}
\vpi^{\frac{1}{|g|}(i + \frac{1-|g|}{2})}.
\]
Define $\vphi^{*}_{g} \in \Map(G, L^{c})$ by
\[
\vphi^{*}_{g}(\gamma) =
\begin{cases}
\sigma^{i}(\beta^{*}_{g})  &\text{if $\gamma = g^i$;} \\
0 &\text{if $\gamma \notin \langle g \rangle$}.
\end{cases}
\]
Then
\begin{equation} \label{E:tlocres}
\br_G(\vphi^{*}_{g}) = \sum_{i=0}^{|g|-1} \vphi_{g}(g^i) g^{-i} =
\sum_{i=0}^{|g|-1} \sigma^{i}(\beta^{*}_{g}) g^{-i}.
\end{equation}
\end{definition}

\begin{theorem} \lab{T:tstickfac}
(cf. \cite[Theorem 7.9]{AM})
If $a_{nr}$ is any choice of n.i.b. generator of $L_{\pi_{nr}}/L$,
then the element $b$ of $L_{\pi}$ defined by
\be \lab{E:sifac}
\br_G(b) = \br_G(a_{nr}) \cdot \br_G(\vphi^{*}_{s})
\ee
satisfies $A_\pi = O_{L}G \cdot b$.
\end{theorem}

\begin{proof}
To ease notation, set $N:= L^{nr}$ and $H := \langle s \rangle$.

Write $[\wt{\pi}] \in H^1(N, G)$ for the image of $[\pi]
\in H^1(L, G)$ under the restriction map $H^1(L, G) \to
H^1(N, G)$. Then $A_{\wt{\pi}} = O_{N}\cdot A_\pi$,
because $N/L$ is unramified. Hence, to establish the desired result,
it suffices to show that

\be \lab{E:r1}
A_{\wt{\pi}} = O_{N}G \cdot b.
\ee
As $\br_G(a_{nr}) \in (O_{N}G)^{\times}$, 
\eqref{E:r1} 
is equivalent to the equality
\be \lab{E:r2}
A_{\wt{\pi}} = O_{N}G \cdot \vphi^{*}_{s}.
\ee

Now
\be \lab{E:bwiso}
N_{\wt{\pi}} \simeq \prod_{H \backslash G} N^{\wt{\pi}},
\ee
where $N^{\wt{\pi}} = N(\vpi^{1/|s|})$ (cf. \eqref{E:wiso}), and this
isomorphism induces a decomposition 
\be \lab{E:swiso}
A_{\wt{\pi}} = \prod_{H \backslash G} A^{\wt{\pi}},
\ee
where
\[
A^{\wt{\pi}} = A(N^{\wt{\pi}}) = \vpi^{(1-|s|)/2|s|} \cdot O_{N}
\]
is the square root of the inverse different of the extension $N^{\wt{\pi}}/N$. 

It therefore follows from the definition of $\vphi^{*}_{s}$ that
\eqref{E:r2} holds if and only if 
\be \lab{E:r3}
A^{\wt{\pi}} = O_{N}H \cdot \beta^{*}_{s}.
\ee
This last equality follows exactly as in \cite[Proposition
  4.2.2]{T16}, and a proof is given by taking $n = (1-e)/2$ (for $e$
odd) in Example \ref{E:ke} above. 
\end{proof}

\begin{proposition} \lab{P:rep}
Suppose that $[\pi] \in H^{1}_{t}(L,G)$ and that $s := \pi(\sigma)$ is
of odd order. Then the class
\[
\fc(\pi) := [A_\pi, O_LG;\br_G] - [O_{\pi}, O_LG;\br_G] \in K_0(O_LG,
L) \simeq \Det(LG)^{\times}/\Det(O_LG)^{\times}
\]
is represented by $\Det(\br_G(\vphi^{*}_{s})) \cdot
\Det(\br_G(\vphi_{s}))^{-1} \in \Det(LG)^{\times}$.
\end{proposition}

\begin{proof}
This is a direct consequence of Theorems \ref{T:stickfac} and
\ref{T:tstickfac}, together with the proof of Proposition \ref{P:des}(c).
\end{proof}


\section{Stickelberger pairings and resolvends} \lab{S:stick}

Our goal in this section is to describe explicitly the elements
$\Det(\br_G(\vphi_s))$ and $\Det(\br_G(\vphi^{*}_{s}))$ constructed in
the previous section. We begin by recalling the definition of two
Stickelberger pairings. The first of these is due to L. McCulloh,
while the second is due to C. Tsang in the case of abelian $G$. See
\cite[Definition 9.1]{AM} and \cite[Definition 2.5.1]{T16}.

\begin{definition} \lab{D:sticks}
Let $\zeta = \zeta_{|G|}$ be a fixed, primitive, $|G|$-th root of
unity. Suppose first that $G$ is cyclic. For $g \in G$ and $\chi \in
\Irr(G)$, write $\chi(g) = \zeta^r$ for some integer $r$.
\begin{itemize}
\item[(1)] We define 
\[
\langle \chi, g \rangle_G = \{r/|G|\},
\]
where $ 0 \leq \{r/|G|\} <1$ denotes the fractional part of $r/|G|$.

Alternatively (cf. Example \ref{E:ke}), if we choose $r$ to be the
unique integer in the set $\{l : 0 \leq l \leq |G|-1\}$ such that
$\chi(g) = \zeta^r$, then
\[
\langle \chi, g \rangle_G = r/|G|.
\]

\item[(2)] Suppose that $|G|$ is odd, and choose $r \in [(1-|G|)/2,
  (|G|-1)/2]$ to be the unique integer such that $\chi(g) =
\zeta^r$. Define
\[
\langle \chi, g \rangle^*_G = r/|G|.
\]

We extend each of these to pairings
\[
\Q R_G \times \Q G \to \Q
\]
via linearity. Finally, we extend the definitions to arbitrary finite
groups $G$ by setting
\[
\langle \chi, s \rangle _G := \langle \chi \mid_{\langle s \rangle}, s
\rangle_{\langle s \rangle}
\]
and
\[
\langle \chi, s \rangle ^*_G := \langle \chi \mid_{\langle s \rangle}, s
\rangle^{*}_{\langle s \rangle},
\]
where the second definition of course only makes sense when the order
$|s|$ of $s$ is odd. \qed
\end{itemize}
\end{definition}

We shall make use of the following alternative descriptions of the
above Stickelberger pairing using the standard inner product on $R_G$
(see \cite[Proposition 9.2]{AM}). For each element $s \in G$, write
$\zeta_{|s|} = \zeta_{|G|}^{|G|/|s|}$, and define a character $\xi_s$
of $\langle s \rangle$ by $\xi_{s}(s^i) = \zeta^{i}_{|s|}$. Set
\[
\Xi_s := \frac{1}{|s|} \sum_{j=1}^{|s|-1} j \xi^{j}_{s}.
\]
For $|s|$ odd, we also define
\[
\Xi^{*}_{s} := \frac{1}{|s|} \sum_{j=1}^{(|s|-1)/2} j(\xi^{j}_{s} -
\xi^{-j}_{s}).
\]

Let $(-,-)_G$ denote the standard inner product on $R_G$.

\begin{proposition} \lab{P:gd}\
\begin{itemize}
\item[(a)] If $s \in G$ and $\chi \in R_G$, we have
\[
\langle \chi, s \rangle_G = (\Ind^{G}_{\langle s \rangle} (\Xi_{s}),
\chi)_G.
\]

\item[(b)] If furthermore $|s|$ is odd, then we have
\[
\langle \chi, s \rangle^*_G = (\Ind^{G}_{\langle s \rangle} (\Xi^{*}_{s}),
\chi)_G.
\]

\item[(c)] If $|s|$ is odd, then
\[
\Xi^{*}_{s} - \Xi_s = -\sum_{j = 1}^{(|s|-1)/2} \xi^{-j}_{s}.
\]

\item[(d)] For $s$ odd, write 
\[
d(s) := -\sum_{j = 1}^{(|s|-1)/2} \xi^{-j}_{s}.
\]
Then, for each $\chi \in R_G$, we have
\[
\langle \chi, s \rangle^{*}_{G} - \langle \chi, s \rangle_G =
(\Ind^{G}_{\langle s \rangle} (d(s)), \chi)_G.
\]
\end{itemize}
\end{proposition}

\begin{proof}
Part (a) is proved in \cite[Proposition 9.2]{AM}. The proof of (b)
is the same \textit{mutatis mutandis}. Part (c) follows directly from
the definitions of $\Xi_s$ and $\Xi^{*}_{s}$, and then (d) follows
from (a) and (b).
\end{proof}

We may use Proposition \ref{P:gd} to describe the relationship between
the two Stickelberger pairings in Definition \ref{D:sticks} when $|s|$ is odd. 

In the sequel, for any finite group $\Gamma$ (which will be clear from
context), and any natural number $k$, we write $\psi_k$ for the $k$-th
Adams operator on $R_{\Gamma}$. Thus, if $\chi \in R_{\Gamma}$ and
$\gamma \in \Gamma$, then one has $\psi_k(\chi)(\gamma) =
\chi(\gamma^k)$.  
In particular, we recall that, for all $k$,
 $\psi_{k}$ commutes with the restriction and inflation functors, as
 well as with the action of $\Omega_{\bQ}$ on $R_\Gamma$ (see
 \cite[Proposition-Definition 3.5]{E91}).
If $L$ is a number field
or a local field, we also write $\psi_k$ for the homomorphism
\[
\Hom(R_{\Gamma}, (L^{c})^{\times}) \to \Hom(R_{\Gamma},
(L^c)^{\times})
\]
defined by setting
\[
\psi_k(f)(\chi) = f(\psi_k(\chi))
\]
for $f \in \Hom(R_{\Gamma}, (L^{c})^{\times})$ and $\chi \in R_{\Gamma}$.

\begin{proposition} \lab{P:adams}
Suppose that $s \in G$ is of odd order, and set $H:= \langle s \rangle$.
\begin{itemize}
\item[(a)] If $1 \leq j \leq |s|-1$, then
\begin{align*}
(\Xi^{*}_{s}, \xi^j )_H &= (\Xi_{s}, \xi^{2j} -
\xi^j )_H \\
&= (\Xi_{s}, \psi_2(\xi^{j}) - \xi^j)_H.
\end{align*}

\item[(b)] (C. Tsang) For each $\chi \in R_G$, we have
\[
\langle \chi, s\rangle^*_G = \langle \psi_2(\chi) - \chi, s \rangle_G.
\]
\end{itemize}
\end{proposition}

\begin{proof}
(a) If $1 \leq j \leq |s|/2$, then we have
\[
(\Xi_{s}, \xi_{s}^{2j} - \xi_{s}^{j})_H = \frac{2j-j}{|s|} = \frac{j}{|s|},
\]
while if $|s|/2 \leq j \leq s-1$, then
\[
(\Xi_{s}, \xi_{s}^{2j} - \xi_{s}^{j} )_H = \frac{(2j - |s|) -j}{|s|} = \frac{j - |s|}{|s|}. 
\]
Thus in each case we have
\[
(\Xi^{*}_{s}, \xi_{s}^{j} )_H = (\Xi_{s}, \xi_{s}^{2j} -
\xi_{s}^{j} )_H,
\]
and this establishes the claim.

(b) Proposition \ref{P:gd}(b), together with Frobenius reciprocity, gives
\begin{align*}
\langle \chi, s \rangle^*_G &= (\Ind^{G}_{\langle s \rangle} (\Xi^{*}_{s}),
\chi)_G \\
&= (\Xi^{*}_{s}, \chi \mid_H)_H.
\end{align*}
The desired result now follows from part (a), together with the fact
that, for any $\chi \in R_G$, we have the equality
\[
\psi_2(\chi) \mid_{H} = \psi_2(\chi \mid_H).
\]
\end{proof}

The following result describes the elements $\Det(\br_G(\vphi_{s}))$
and $\Det(\br_G(\vphi^{*}_{s}))$ in terms of Stickelberger pairings.
In what follows, we retain the notation and conventions of Section
\ref{S:loc}.

\begin{proposition} \lab{P:protojac}
Suppose that $\chi \in R_G$ and $s \in G$.

\begin{itemize}
\item[(a)] We have
\[
\Det(\br_G(\vphi_{s}))(\chi) = \vpi^{\langle \chi, s \rangle_G}.
\]

\item[(b)] If $|s|$ is odd, then we have
\[
\Det(\br_G(\vphi^{*}_{s}))(\chi) = \vpi^{\langle \chi, s \rangle^{*}_{G}}.
\]

\item[(c)] For $|s|$ odd, we have
\begin{align*}
[\Det(\br_G(\vphi^{*}_{s})) \cdot \Det(\br_G(\vphi_s))^{-1}] 
(\chi) &= \vpi^{\langle \chi, s
  \rangle^{*}_{G} -\langle \chi, s \rangle_{G}}\\
&=\vpi^{\langle \psi_{2}(\chi) - 2\chi, s \rangle_{G}}\\
&=
\frac{\Det(\br_G(\vphi_{s}))(\psi_2(\chi))}{\Det(\br_G(\vphi_{s}))(2\chi)}.
\end{align*}
That is to say,
\[
\Det(\br_G(\vphi^{*}_{s})) \cdot \Det(\br_G(\vphi_s))^{-1} =
\psi_2(\Det(\br_G(\vphi_{s}))) \cdot \Det(\br_G(\vphi_{s}))^{-2}.
\]
\end{itemize}
\end{proposition}

\begin{proof}
Part (a) is proved in \cite[Proposition 10.5(a)]{AM}. The proof of
(b) is very similar, using \cite[Proposition 4.2.2]{T16}, which in
fact shows the result for $G$ abelian. Part (c) follows from parts (a)
and (b), and Proposition \ref{P:adams}. 
\end{proof}

\begin{corollary} \lab{C:rephom}
Suppose that $[\pi] \in H^1_t(L,G)$, and that $s:= \pi(\sigma)$ is of
odd order. Then a representing homomorphism for the class
\[
\fc(\pi) = [A_\pi, O_LG;\br_G] - [O_{\pi}, O_LG; \br_G]
\]
in 
\[
K_0(O_LG, L) \simeq \frac{\Det(LG)^{\times}}{\Det(O_LG)^{\times}} 
\simeq \frac{\Hom_{\Omega_{L}}(R_G,
  (L^{c})^{\times})}{\Det(O_LG)^{\times}}
\]
is the map $f_{\pi} \in \Hom_{\Omega_{L}}(R_G, (L^{c})^{\times})$
given by
\[
f_{\pi}(\chi) = \vpi^{\langle \psi_{2}(\chi) - 2\chi, s \rangle_{G}}.
\]
\end{corollary}

\begin{proof}
This follows from Propositions \ref{P:rep} and \ref{P:protojac}(c).
\end{proof}




\section{Galois-Gauss and Galois-Jacobi sums} \lab{S:sum}

Let $L$ be a local field of residual characteristic $p$. Suppose that
$[\pi] \in H^1_t(L, G)$, and recall that we have (see \eqref{E:wiso})
\[
L_{\pi} \simeq \prod_{\pi(\Omega_L)\backslash G} L^{\pi}.
\]
Set $H := \pi(\Omega_L) = \Gal(L^{\pi}/L)$, and write
$\tau^{*}(L^{\pi}/L,\, -) \in \Hom(R_{H}, (\bQ^c)^{\times})$ for the
adjusted Galois-Gauss sum homomorphism associated to $L^{\pi}/L$ (see
\cite[Chapter IV, (1.7)]{Fr83}). We define $\tau^{*}(L_{\pi}/L,\, -)
\in \Hom(R_{G}, (\bQ^c)^{\times})$ by composing $\tau^{*}(L^{\pi}/L,\,
-)$ with the natural map $R_G \to R_H$.


For a finite group $\Gamma$, we write $\Irr_p(\Gamma)$ for the set of 
$\bQ_{p}^{c}$-valued irreducible characters of $\Gamma$ and $R_{\Gamma, p}$ 
for the free abelian group on $\Irr_p(\Gamma)$.
We fix a local embedding $\Loc_p: \bQ^c \to \bQ^c_p$, and we
shall identify $\Irr(\Gamma)$ with $\Irr_p(\Gamma)$ via this choice
of embedding.

For each rational prime $l \neq p$, we fix a semi-local embedding
$\Loc_l: \bQ^c \to (\bQ^c)_l := \bQ^c \otimes_{\bQ} \bQ_{l}$. (Caveat:
note that this is not the same thing as a local embedding $\bQ^c \to
\bQ^c_l$!) For each rational prime $l$, write $\bQ^t_l$ for the
maximal, tamely ramified extension of $\bQ_l$.  

We shall require the following results.

\begin{proposition}\lab{P:cn-t} 
Fix a rational prime $l$.

\begin{itemize}
\item[(a)] Let $K$ be an unramified extension of $\bQ_{l}$. Then, for
  any integer $k$, we have that
\[
\psi_k(\Det(O_{K}G)^{\times}) 
\subseteq \Det(O_{K}G)^{\times}.
\]	
\item[(b)] Let $\Gamma$ be a finite group with abelian $p$-Sylow
  subgroups. Then, for any integer $k$,
\[
\psi_k(\Det(O_{\bQ_{p}^{t}} \Gamma)^{\times}) \subseteq 
\Det(O_{\bQ_{p}^{t}} \Gamma)^{\times}.
\]
\item[(c)] Suppose that $l \neq p$. Then
\[
\Loc_l(\tau^{*}(L_{\pi}/L,\, -)) \in 
\Det(O_{\bQ(\mu_p),l} G)^{\times}.
\]
\end{itemize}	
\end{proposition}
\begin{proof}
Parts (a) and (b) are results of Cassou-Nogu\`es and Taylor.  For part (a)
see, e.g. \cite[Chapter 9, Theorem 1.2]{Ty84}, and note that for this
particular result we do not need to assume that $(k, |G|) = 1$. For part
(b) see \cite[pp. 83, Remark]{CN-T98}.

Part (c) follows from \cite[Chapter IV, Theorems 30]{Fr83},
where analogous results are proved for $\tau^{*}(L^{\pi}/L,\,-)$; the
corresponding results for $\tau^{*}(L_{\pi}/L,\,-)$ are then a direct
consequence of the definition of $\tau^{*}(L_{\pi}/L,\,-)$. 
\end{proof}

The following result is entirely analogous to \cite[Chapter IV, Lemma
  2.1]{Fr83}. Recall that if $f \in \Hom(R_{\Gamma},
(\bQ^{c}_{p})^{\times})$, then $\omega \in \Omega_{\bQ_{p}}$ acts on
$f$ by the rule
\[
f^{\omega}(\chi) = f(\chi^{\omega^{-1}})^{\omega}.
\]

\begin{lemma} \lab{L:unram}
Let $L/\bQ_p$ be a finite extension, and let $\{\nu \}$ be any right
transversal of $\Omega_L$ in $\Omega_{\bQ_{p}}$. Suppose that $f \in
\Hom_{\Omega_{L^{\nr}}}(R_{\Gamma}, (\bQ_{p}^{c})^{\times})$. Then
(cf. \eqref{E:normk} and \eqref{E:normc}):
\[
\cN_{L/\bQ_{p}} f := \prod_{\nu} f^{\nu} \in
\Hom_{\Omega_{\bQ_{p}^{\nr}}}(R_{\Gamma}, (\bQ_{p}^{c})^{\times}).
\]
\end{lemma}

\begin{proof}
It suffices to show that this result holds with respect to a
particular choice of transversal of $\Omega_L$ in $\Omega_{\bQ_{p}}$.

We first observe that, as $\Omega_{\bQ^{\nr}_{p}}$ is normal in
$\Omega_{\bQ_{p}}$, $\Omega_L \cdot \Omega_{\bQ^{\nr}_{p}}$ is a
subgroup of $\Omega_{\bQ_{p}}$. We choose a right transversal
$\{\omega\}$ of $\Omega_L \cdot \Omega_{\bQ^{\nr}_{p}}$ in
$\Omega_{\bQ_{p}}$.

Next, we choose a right transversal $\{\sigma\}$ of $\Omega_L \cap
\Omega_{\bQ^{\nr}_{p}}$ in $\Omega_{\bQ^{\nr}_{p}}$. It follows that
$\{\sigma\}$ is also a right transversal of $\Omega_L$ in $\Omega_L
\cdot \Omega_{\bQ^{\nr}_{p}}$. We now deduce that $\{\sigma \omega \}$
is a right transversal of $\Omega_L$ in $\Omega_{\bQ_{p}}$. We also
note that
\[
\Omega_{L} \cap \Omega_{\bQ^{\nr}_{p}} = \Omega_{L^{\nr}} \cap
\Omega_{\bQ^{\nr}_{p}},
\]
and that (since $\Omega_{\bQ^{\nr}_{p}}$ is normal in
$\Omega_{\bQ_{p}}$), 
\[
\omega^{-1}_{i} (\Omega_{L^{\nr}} \cap \Omega_{\bQ^{\nr}_{p}})
\omega_i = 
\omega^{-1}_{i} \Omega_{L^{\nr}} \omega_i \cap
\Omega_{\bQ^{\nr}_{p}}
\]
for any $\omega_i \in \{\omega\}$.

Now suppose that $f \in \Hom_{\Omega_{L^{\nr}}}(R_{\Gamma},
(\bQ_{p}^{c})^{\times})$ and that $\omega_i \in \{\omega\}$. Then
\[
f^{\omega_{i}} \in \Hom_{\omega^{-1}_{i}\Omega_{L^{\nr}} \omega_i}(R_{\Gamma},
(\bQ_{p}^{c})^{\times}),
\]
and so 
\[
f^{\omega_{i}} \in \Hom_{(\omega^{-1}_{i}\Omega_{L^{\nr}} \omega_{i}) \cap
\Omega_{\bQ^{nr}_{p}}}(R_{\Gamma},
(\bQ_{p}^{c})^{\times}).
\]

Now observe that, for fixed $\omega_i \in \{\omega\}$,
$\{\omega^{-1}_{i} \sigma \omega_i\}_{\sigma}$ is a right transversal
of $\omega^{-1}_{i}\Omega_{L^{\nr}} \omega_i \cap
\Omega_{\bQ^{nr}_{p}}$ in $\Omega_{\bQ^{\nr}_{p}}$, and so
\[
\prod_{\sigma} (f^{\omega_{i}})^{\omega^{-1}_{i} \sigma \omega_{i}}
\in \Hom_{\Omega_{\bQ^{\nr}_{p}}}(R_{\Gamma}, (\bQ^{c}_{p})^{\times}).
\]

Hence finally we obtain 
\[
\prod_{\omega, \sigma} (f^{\omega})^{\omega^{-1}\sigma \omega} =
\prod_{\omega, \sigma} f^{\sigma \omega} \in
\Hom_{\Omega_{\bQ^{\nr}_{p}}}(R_{\Gamma}, (\bQ^{c}_{p})^{\times}),
\]
as required.
\end{proof}

\begin{proposition}\lab{P:gg}
Let $a_{\pi}$ be any n.i.b. generator of
$L_{\pi}/L$. Suppose also that the square root $A_\pi$ of the inverse
different of $L_{\pi}/L$ exists (i.e. that $s := \pi(\sigma)$ is of
odd order), and that $A_\pi = O_L G \cdot b_{\pi}$. 
Then:
\begin{itemize}
\item[(a)] 
$
\cN_{L/\bQ_{p}}[ \Det(\br_G(b_{\pi}))^{-1} \cdot
\psi_2(\Det(\br_G(a_{\pi}))) \cdot \Det(\br_G(a_{\pi}))^{-1}]
\in \Det(O_{\bQ^{t}_{p}}G)^{\times}.
$
\item[(b)] 
\begin{itemize}
\item[(i)] $\Loc_p[(\tau^{*}(L_{\pi}/L,\, -))]^{-1} \cdot
\cN_{L/\bQ_{p}}[\Det(\br_{G}(a_{\pi})) ] \in  \Det(O_{\bQ^{t}_{p}}
G)^{\times}$.
\item[(ii)] $\Loc_p[\psi_2(\tau^{*}(L_{\pi}/L,\,-))]^{-1} \cdot
\cN_{L/\bQ_{p}}[\psi_2(\Det(\br_{G}(a_{\pi})))]\in
\Det(O_{\bQ^{t}_{p}} G)^{\times}$.
\end{itemize}

%
\item[(c)] 
$\Loc_p[\psi_2(\tau^{*}(L_{\pi}/L,\, -)) \cdot (\tau^{*}(L_{\pi}/L,\, -))^{-1}]^{-1} \cdot
\cN_{L/\bQ_p}[\Det(\br_{G}(b_{\pi}))] \in
\Det(O_{\bQ^{t}_{p}} G)^{\times}.$
\item[(d)] 
$
\Loc_p[\psi_2(\tau^{*}(L_{\pi}/L,\,-)) \cdot (\tau^{*}(L_{\pi}/L,\, -))^{-2}]^{-1} \cdot
\cN_{L/\bQ_p}[\Det(\br_{G}(b_{\pi})) \cdot \Det(\br_{G}(a_{\pi}))^{-1}] 
$\\
 belongs to $\Det(O_{\bQ^{t}_{p}} G)^{\times}$.
\item[(e)] With the notation of Proposition~\ref{P:rep}, the element
\[
\Loc_p[\psi_2(\tau^{*}(L_{\pi}/L,\,-)) \cdot (\tau^{*}(L_{\pi}/L,\, -))^{-2}]^{-1} \cdot
\cN_{L/\bQ_p}[\Det(\br_{G}(\vphi^{*}_{s})) \cdot \Det(\br_{G}(\vphi_{s}))^{-1}] 
\]
belongs to $\Det(O_{\bQ^{t}_{p}} G)^{\times}$.
\end{itemize}
\end{proposition}

\begin{proof}

(a) Recall from \cite[Definition 7.12]{AM} that, for any n.i.b.
  generator $a_{\pi}$ of $L_{\pi}/L$, one has
\[ 
\br_{G}(a_{\pi}) = u\cdot \br_{G}(a_{nr}) \cdot \br_{G}(\vphi_{s}), 
\]
where $u \in (O_LG)^{\times}$ and $\br_{G}(a_{nr})\in
(O_{L^{nr}}G)^{\times}$. Furthermore, $u\cdot a_{nr}$ is also a n.i.b
generator of $L_{\pi_{nr}}/L$.

Hence
\[
\Det(\br_G(a_{\pi}) \cdot \br_G(\vphi_{s})^{-1}) = \Det(u \cdot a_{\nr})
\in
\Det(O_{L^{\nr}}G)^{\times},
\]

and Lemma \ref{L:unram} implies that also
\[
\cN_{L/\bQ_{p}}[\Det(\br_G(a_{\pi}) \cdot \br_G(\vphi_{s})^{-1})]
\in
\Det(O_{\bQ_{p}^{nr}}G)^\times,
\]

It now follows from Proposition \ref{P:cn-t} that the product
\begin{equation} \lab{E:prod}
\cN_{L/\bQ_{p}}[ (\Det(\br_{G}(a_{\pi})) \cdot \Det(
\br_{G}(\vphi_{s}))^{-1})^{-1} \cdot \psi_{2}
\bigl(\Det(\br_{G}(a_{\pi})) \cdot \Det(
\br_{G}(\vphi_{s}))^{-1}\bigr) ]
\end{equation}
belongs to $\Det(O_{\bQ_{p}^{nr}}G)^\times$.

%


Part (a) now  follows from \eqref{E:prod}, together with
Proposition~\ref{P:protojac}(c) and the Stickelberger factorisation of
$\br_G(b_{\pi})$ (see Theorem \ref{T:tstickfac}).
\medskip

(b) Let $O^\pi$ denote the integral closure of $O_L$
in $L^\pi$ and fix an element $\alpha \in L^{\pi}$ such that $O^{\pi}
= O_LH \cdot \alpha$.  It follows from \cite[Chapter IV, Theorem 31]{Fr83} that
there exists an element $w\in (O_{\bQ_{p}^{t}}H)^\times$ such that
\begin{equation}\label{local to G eq 1}
{\rm Loc}_{p}(\tau^{\ast} (L^{\pi}/L, - ))^{-1} \cdot 
\cN_{L/\bQ_{p}} \Det(\br_{H}(\alpha)) = \Det(w) 
\end{equation}
Under our hypotheses, the inertia subgroup of $H$ is cyclic of order
$|s|$ coprime to $p$. Hence Proposition~\ref{P:cn-t}(b) implies that
\begin{equation}\label{local to G eq 2}
{\rm Loc}_p[\psi_2(\tau^{*}(L^{\pi}/L,\,-))]^{-1} \cdot
\cN_{L/\bQ_{p}}[\psi_2( \Det(\br_{H}(\alpha)) )]
\end{equation}
belongs to $\psi_{2}(\Det(O_{\bQ_{p}^{t}}H)^\times ) \subseteq
\Det(O_{\bQ_{p}^{t}}H)^\times \subseteq \Det(O_{\bQ_{p}^{t}}G)^\times$.

Next, we construct a map $a_{\pi} \in {\rm Map}(G, L^c)$ associated to
$\alpha$ by setting 
\[ a_{\pi}(\gamma) := \left\{ \begin{array}{ll}
\tilde{\gamma}(\alpha), & \quad \mbox{if } \quad \gamma=\pi(\tilde{\gamma}) \text{ for } 
\tilde{\gamma} \in \Omega_{L} ;\\
0, & \quad \mbox{otherwise}.\end{array} \right. \] 
It is easy to see from \eqref{E:wiso} that $a_{\pi} \in L_\pi$ and satisfies that
$O_{\pi} = O_LG \cdot a$. In particular, for each $\chi \in R_{G}$, we have
\[
\Det_\chi( {\bf r}_G(a_{\pi})) 
= \Det_{\chi} \bigl(   \sum_{\gamma \in G} a_{\pi}(\gamma)\gamma^{-1}  \bigr) 
= \Det_{\chi} \bigl(  \sum_{\gamma\in H} \tilde{\gamma}(\alpha)\gamma^{-1} \bigr) 
=  \Det_{{\rm res} \chi}(\br_{H}(\alpha)), \]
with ${\rm res}:= {\rm res}^{G}_H: R_+{G} \to R_{H}$. This implies that
\begin{equation}\label{local to G eq 3}
\begin{array}{rl}
\cN_{L/\bQ_{p}}[\Det(\br_{G}(a_{\pi}))]  =\, & \cN_{L/\bQ_{p}}[\Det(\br_{H}(\alpha))] ,\\
\cN_{L/\bQ_{p}}[\psi_2(\Det(\br_{G}(a_{\pi})))] =\, & \cN_{L/\bQ_{p}}[\psi_2(\Det(\br_{H}(\alpha)))] .\\
\end{array}
\end{equation}

We now see from the definition of $\tau^{\ast}(L_{\pi}/L, - )$ that
(i) follows from \eqref{local to G eq 1}, \eqref{local to G eq 3},
while part (ii) is a consequence of \eqref{local to G eq 2} and
\eqref{local to G eq 3}.
\medskip

(c) Follows from (a) and (b) above.
\medskip

(d) Follows from (b)(i) together with (c).
\medskip

(e) Follows from (d) above.

\end{proof}

Proposition \ref{P:gg}(d) and (e)
motivate the following definition.

\begin{definition}
We retain the notation established above. Define the \textit{adjusted
  Galois-Jacobi sum homomorphism associated to $L_{\pi}/L$},
$J^*(L_{\pi}/L,\,-) \in
\Hom(R_{G}, (\bQ^c)^{\times})$, by
\[
J^*(L_{\pi}/L,\,-) := \psi_2(\tau^*(L_{\pi}/L,\,-)) \cdot (\tau^*(L_{\pi}/L,\,-))^{-2}.
\]

It follows from the Galois action formulae for Galois-Gauss sums (see
\cite[pp. 119 and 152]{Fr83}) that in fact
$J^*(L_{\pi}/L,\,-) \in \Hom_{\Omega_{\bQ}}(R_{\Gamma}, (\bQ^c)^{\times})$.
\qed
\end{definition}

\begin{remark} \lab{R:AF}
Let $\tau(L^{\pi}/L,\,-) \in \Hom(R_{H}, (\bQ^{c})^{\times})$ denote
the (unadjusted) Galois-Gauss sum associated to $L^{\pi}/L$, and write
$\tau(L_{\pi}/L,\,-) \in \Hom(R_{G}, (\bQ^{c})^{\times})$ for the
composition of $\tau(L^{\pi}/L,\,-)$ with the natural map $R_G \to
R_H$.  We remark that the Galois-Jacobi sum $J(L_{\pi}/L,\, -) \in
\Hom(R_{G}, (\bQ^c)^{\times})$ defined by
\[
J(L_{\pi}/L,\, -) := \psi_2(\tau(L_{\pi}/L,\, -)) \cdot
(\tau(L_{\pi}/L,\, -))^{-2}
\]
is a special case of the non-abelian Jacobi sums first introduced by
A. Fr\"ohlich (see \cite{Fr77}).  
\qed
\end{remark}

\begin{proposition} \lab{P:jkill}\
\begin{itemize}
\item[(a)] Suppose that $l \neq p$. Then
\[
\Loc_l(J^*(L_{\pi}/L, -)) \in \Det(\bZ_l G^{\times}).
\]

\item[(b)] Using the notation of Proposition \ref{P:gg}, we have
\[
\Loc_p(J^*(L_{\pi}/L,\, -))^{-1} \cdot
\cN_{L/\bQ_p}[\Det(\br_{G}(b_{\pi})) \cdot \Det(\br_{G}(a_{\pi}))^{-1}] 
\in \Det(\bZ_p G^{\times}).
\]
Hence
\[
\Loc_p(J^*(L_{\pi}/L,\, -))^{-1} \cdot
\cN_{L/\bQ_p}[\Det(\br_{G}(\vphi^{*}_{s})) \cdot \Det(\br_{G}(\vphi_s))^{-1}] 
\in \Det(\bZ_p G^{\times}).
\]
\end{itemize}
\end{proposition}

\begin{proof}
(a) Recall that $J^*(L_{\pi}/L,\, -) \in \Hom_{\Omega_{\bQ}}(R_{G},
(\bQ^c)^{\times})$, and that $\bQ(\mu_p)/\bQ$ is unramified at
$l$. It therefore follows from Proposition~\ref{P:cn-t} (a) and (c),
together with Taylor's fixed point theorem for determinants 
(see \cite[Chapter 8, Theorem 1.2]{Ty84}), that
\[
\Loc_l(J^*(L_{\pi}/L,\, -)) \in [\Det(O_{\bQ_{l}(\mu_p)} G^{\times})]^{\Omega_{\bQ_{l}}}
= \Det (\bZ_l G^{\times}),
\]
as claimed.

(b) As both of the functions $\Loc_p(J^*(L_{\pi}/L,\,-))$ and $\cN_{L/\bQ_p}[\Det(\br_{G}(b_{\pi}))
  \cdot \Det(\br_{G}(a_{\pi}))^{-1}]$ lie in
$\Hom_{\Omega_{\bQ_{p}}}(R_{G}, (\bQ^{c}_{p})^{\times})$, we see
from Proposition \ref{P:gg}(d) that
\[
\Loc_p(J^*(L_{\pi}/L,\, -))^{-1} \cdot
\cN_{L/\bQ_p}[\Det(\br_{G}(b_{\pi})) \cdot \Det(\br_{G}(a_{\pi}))^{-1}] 
\in [\Det(O_{\bQ^{t}_{p}} G^{\times})]^{\Omega_{\bQ_{p}}} =
\Det(\bZ_p G^{\times}).
\]
The final assertion now follows at once from the Stickelberger
factorisations of $\br_G(a_\pi)$ and $\br_G(b_\pi)$ (see Theorems
\ref{T:stickfac} and \ref{T:tstickfac}).
\end{proof}


\section{Symplectic Galois-Jacobi sums I} \label{symp sec}

In this section we fix data $L, G$ and $\pi$ as in Section \ref{S:sum}.  We
write ${\rm Symp}(G)$ for the set of irreducible symplectic characters
of $G$. For each $\chi \in \Irr(G)$, we write $\tau(L_\pi/L, \chi)$
for the associated (unadjusted) Galois-Gauss sum, and
\[
J(L_{\pi}/L,\, -) := \psi_2(\tau(L_{\pi}/L,\, -)) \cdot
(\tau(L_{\pi}/L,\, -))^{-2}
\]
for the (unadjusted) Galois-Jacobi sum (see Remark \ref{R:AF}).

We shall prove the following result concerning symplectic
  Galois-Jacobi sums.

\begin{theorem}\lab{new addition} 
Suppose that $\chi \in \Symp(G)$. Then $J(L_\pi/L, \chi)$
is a strictly positive real number.
\end{theorem}

We see from the decomposition \eqref{E:wiso} that it is enough to
prove this result after replacing the Galois algebra $L_\pi$ by the
field $L^{\pi}$ and the group $G$ by the Galois group $\pi(\Omega_L) =
\Gal(L^{\pi}/L)$. In the sequel, we shall therefore restrict to the
case that $L_\pi/L$ is a finite Galois extension of $p$-adic fields
and $G$ is its Galois group.
%


 
To prove Theorem ~\ref{new addition}, it is therefore enough to show
that for each $\chi$ in ${\rm Symp}(G)$ the quotient $\tau(L,
\psi_{2}(\chi))/\tau(L, 2\chi)$ is a strictly positive real number.

To verify this, we recall that, since each such $\chi$ is real-valued,
the definition of the local root number $W(L,\chi)$ implies that 
\[\tau(L, \chi) = W(L, \chi) \cdot {\bf N}_{L}\mathfrak{f}(L_\pi/L, \chi)^{1/2}.\]
(cf. \cite[Chapter  II, Section 4, Definition]{M}). Hence, since ${\bf N}_{L}\mathfrak{f}(L_\pi/L, \chi)^{1/2}>0$, it is enough to prove the following result. 


\begin{theorem}\lab{root number theorem} 
Let $E/F$ be a tamely ramified Galois extension of non-archimedean
local fields that has odd ramification degree and set $G :=
\Gal(E/F)$. Then for each $\chi$ in ${\rm Symp}(G)$ one has
$W(F,\psi_2(\chi)) = W(F,2\chi) = 1$. 
\end{theorem}

This sort of result is, in principle, hard to prove both because root
numbers of symplectic characters are difficult to compute and because
Adams operators do not in general commute with induction functors. We
therefore prove two preliminary results that help address these
problems.

The first of these results is entirely representation-theoretic in nature. 

In the sequel, for any finite group $\Gamma$ and character $\phi$ in $R_\Gamma$, we write ${\rm Tr}(\phi)$ for the real-valued character $\phi + \overline{\phi}$. 
\begin{lemma}\label{psi_2 ind - ind psi_2} Let $\Delta$ be a subgroup of a finite group $\Gamma$, fix a character $\phi$ of $\Delta$ and consider the virtual character 
\[ {\rm I}_{\Gamma}^2(\phi) := \psi_2(\mathrm{Ind}_\Delta^{\Gamma}(\phi)) - \mathrm{Ind}_\Delta^{\Gamma}(\psi_2(\phi)).\]
For elements $\gamma$ and $\delta$ of $\Gamma$, we set $\gamma^\delta:= \delta\gamma\delta^{-1}$.

\begin{itemize}
\item[(a)] Let $\mathcal{T}$ be a set of coset representatives of $\Delta$ in $\Gamma$. Then for every $\gamma\in \Gamma$, one has  
\begin{equation*}
 ({\rm I}_{\Gamma}^2(\phi))(\gamma) = \sum_{\tau}\phi((\gamma^\tau)^2),
\end{equation*}
where the sum runs over all $\tau \in \mathcal{T}$ for which $(\gamma^\tau)^2 \in\Delta$ and $\gamma^\tau \notin\Delta$. 
\item[(b)] If $\Delta$ is a subnormal subgroup of $\Gamma$ of odd index, then ${\rm I}_{\Gamma}^2(\phi) = 0$.     
\item[(c)] Assume $\Gamma$ is a semi-direct product of a supersolvable group by an abelian normal subgroup $\Upsilon$. 
\begin{itemize}
\item[(i)] Then for every irreducible character $\mu$ of $\Gamma$, there exists a subgroup $\Upsilon'$ of $\Gamma$ that contains $\Upsilon$ and  a linear character $\lambda$ of $\Upsilon'$ such that $\mu = {\rm Ind}_{\Upsilon'}^\Gamma(\lambda)$.
\end{itemize}
 In addition, if $\Upsilon \subseteq \Delta$, the index of $\Delta$ in $\Gamma$ is a power of $2$ and $\Gamma$ has cyclic  Sylow $2$-subgroups, then the following claims are also valid.
\begin{itemize}
\item[(ii)] If $\phi$ is real-valued, then ${\rm I}_{\Gamma}^2(\phi)$ is an integral linear combination of characters of the form ${\rm Ind}_{\Delta'}^\Gamma\lambda$ and ${\rm Tr}(\phi')$, where $\Delta'$ runs over subgroups of $\Gamma$ that contain $\Delta$, $\lambda$ over homomorphisms $\Delta' \to \{\pm 1\}$ and $\phi'$ over elements of $R_\Gamma$. 
\item[(iii)] If $\phi$ is induced from a proper normal subgroup of $\Delta$ of $2$-power index that contains $\Upsilon$, then ${\rm I}_{\Gamma}^2(\phi)=0$. 
\end{itemize}
\item[(d)] Assume $\Gamma$ is generalized quaternion, $\Delta$ is the cyclic subgroup of $\Gamma$ of index $2$ and $\phi$ is irreducible (and hence linear). Then $\phi^2$ is trivial on the centre $Z$ of $\Gamma$ and 
\[ \psi_2\bigl({\rm Ind}_\Delta^\Gamma\phi\bigr) = {\rm Inf}_{\Gamma/Z}^\Gamma\bigl({\rm Ind}_{\Delta/Z}^{\Gamma/Z}(\phi^2)\bigr) + {\rm Inf}_{\Gamma/\Delta}^\Gamma(\chi_{\Gamma/\Delta}) - {\bf 1}_\Gamma,\]
where we regard $\phi^2$ as a character of $\Delta/Z$ and write $\chi_{\Gamma/\Delta}$ for the unique non-trivial homomorphism $\Gamma/\Delta \to (\bQ^{c})^{\times}$.  
\end{itemize}
\end{lemma} 

\begin{proof} Part (a) follows directly from the explicit formula for induced characters and the fact that for each $\gamma\in \Gamma$, and $\tau\in \mathcal{T}$ one has $(\gamma^\tau)^2\in \Delta$ whenever $\gamma^\tau\in \Delta$. 

To prove part (b), we fix a chain of subgroups 
\begin{equation}\label{chain} \Delta = \Gamma(1) \subset ... \subset \Gamma(t-1) \subset \Gamma(t)
 = \Gamma\end{equation}
such that each $\Gamma(i)$ is normal in $\Gamma(i+1)$. Then the equality  
\begin{equation}\label{induction formula} 
{\rm I}^2_\Gamma(\phi) = \sum_{i=1}^{i = t-1} {\rm Ind}_{\Gamma(i+1)}^\Gamma \bigl(  {\rm I}_{\Gamma(i+1), \Gamma(i)}^2({\rm Ind}^{\Gamma(i)}_\Delta\phi) \bigr), 
\end{equation}
where
\[{\rm I}_{\Gamma(i+1), \Gamma(i)}^2 (\chi) = \psi_{2}({\rm Ind}^{\Gamma(i+1)}_{\Gamma(i)} \chi ) - {\rm Ind}^{\Gamma(i+1)}_{\Gamma(i)} (\psi_{2}(\chi)), \]
reduces us to the case $\Delta$ is normal in $\Gamma$. In this case, the claim follows immediately from the formula in part (a) and the fact that, under the stated conditions, for every $\gamma\in \Gamma$ and $\tau\in \mathcal{T}$ one has $(\gamma^\tau)^2 \in \Delta \Longleftrightarrow \gamma^\tau \in \Delta$.  

Turning to part (c), we note first that, under the stated hypothesis on $\Gamma$, claim (c)(i)  follows from \cite[Section 8.5, Exercise 8.10]{S1} and the argument of \cite[Section 8.2, Proposition 25]{S1}. 

To verify (c)(ii) and (c)(iii) we assume the additional hypotheses on $\Gamma$ and note, in particular, that since $\Gamma$ has cyclic Sylow $2$-subgroups, Cayley's normal $2$-complement theorem implies that $\Gamma$, and therefore also its quotient $\Gamma/\Upsilon$, has a normal $2$-complement. Writing $\Upsilon_1/\Upsilon$ for the normal $2$-complement of $\Gamma/\Upsilon$, the  given assumptions imply $\Upsilon_1\subseteq \Delta$ and so, since $\Gamma/\Upsilon_1$ is cyclic of $2$-power order, there exists a chain of subgroups (\ref{chain}) in which $\Gamma(i)$ has index $2$ in $\Gamma(i+1)$ for each $i$. The corresponding equality (\ref{induction formula}) then reduces claims (c)(ii) and (c)(iii) to the case that $\Delta$ has index two in $\Gamma$. In this case $|\mathcal{T}| = 2$ and, for every $\gamma\in\Gamma$ and $\tau\in \mathcal{T}$, one has $(\gamma^\tau)^2\in \Delta$ and, in addition, $\gamma^\tau \notin \Delta \Longleftrightarrow \gamma\notin \Delta$ and so the formula in part (a) implies 
\begin{equation}\label{c equality} (I_\Gamma^2(\phi))(\gamma) = \begin{cases} 0, &\text{ if $\gamma\in \Delta$,}\\
                               \sum_{\tau\in \mathcal{T}}\phi((\gamma^\tau)^2), &\text{ if $\gamma\notin \Delta$.}\end{cases}\end{equation}             
Now, by (c)(i), every irreducible character of $\Gamma$ has the form $\mu = {\rm Ind}_{\Upsilon'}^\Gamma(\lambda)$, where $\Upsilon'$ is a suitable subgroup of $\Gamma$ that contains $\Upsilon$ and $\lambda$ a linear character of $\Upsilon'$.  Further, if $\Upsilon'\not\subset \Delta$, then the index of $\Upsilon'$ in $\Gamma$ is odd so $\mu$ has odd degree and so, by \cite[Theorem A]{NST}, is real-valued if and only if it is a homomorphism of the form $\Upsilon' \to \{\pm 1\}$. Claim (c)(ii) follows directly from this fact and the observation that $I_\Gamma^2(\phi)$ is real-valued if $\phi$ is real-valued.

To prove claim (c)(iii), we assume $\phi = {\rm Ind}_{\Delta'}^\Delta\phi'$, where $\Delta'$ is a normal subgroup of $\Delta$ that contains $\Upsilon$ and is of $2$-power index. In this case, the formula (\ref{c equality}) implies that if $I_\Gamma^2(\phi)$ is non-zero, then there exists an element of $\Gamma\setminus \Delta$ whose square belongs to $\Delta'$. However, since $\Upsilon_1 \subseteq \Delta'$, the image in the (cyclic) group $\Gamma/\Delta'$ of any element in $\Gamma\setminus \Delta$ has order divisible by $4$ and so its square cannot belong to $\Delta'$. This proves (c)(iii). 

Next, under the hypotheses of (d), for every $\gamma \in \Gamma$ one has $\gamma^2 \in \Delta$ and hence 
\[ \bigl(\psi_2({\rm Ind}_\Delta^\Gamma\phi)\bigr)(\gamma) = ({\rm Ind}_\Delta^\Gamma\phi)(\gamma^2) = \phi^2(\gamma) + \phi^2(\gamma^{-1}).\]
In particular, since $\phi^2(z) = 1$ for every $z\in Z$, this formula implies $\psi_2({\rm Ind}_\Delta^\Gamma\phi)$ is the  inflation of a character function on the dihedral group $\Gamma/Z$, and then the displayed formula in part (d) is verified by an easy explicit computation.  \end{proof}

In the sequel, for each finite Galois extension $E/F$ of $p$-adic fields, and each complex character $\chi$ of $\Gal(E/F)$ we abbreviate the root number $W(F,\chi)$ to $W(\chi)$. 

Part (c) of the following result relies on the central result of Fr\"ohlich and Queyrut in \cite{FQ73}.

\begin{proposition}\label{fq result} Let $E/F$ be a finite Galois extension of $p$-adic fields. Set 
$G := \Gal(E/F)$ and assume that the inertia subgroup of $G$ has odd order.  
\begin{itemize}
\item[(a)] For all $\phi$ in $R_G$ one has $W({\rm Tr}(\phi)) = 1$. 
\item[(b)] If $H$ is a normal subgroup of $G$ and $G/H$ is cyclic, then for each $\phi$ in $R_H$ one has 
\[ W({\rm Ind}_H^G\phi) = \begin{cases} W(\phi), &\text{if $G/H$ has odd order,}\\
                                        W(\phi)W(\chi_{G/H})^{\phi(1)}, &\text{if $G/H$ has even order,}\end{cases}\]
where, in the second case, $\chi_{E'/F}$ is the non-trivial character of $\Gal(E'/F)$, with $E'$ the quadratic extension of $F$ in $E$.   
\item[(c)] Assume $G$ is dihedral of order congruent to $2$ modulo $4$, write $L$ for the unique quadratic extension of $F$ in $E$ and set $H := \Gal(E/L)$. Then for each homomorphism $\phi: H\to (\bQ^{c})^{\times}$, one has $W({\rm Ind}_H^G\phi) = W(\chi_{G/H})$, where $\chi_{G/H}$ is the non-trivial character of $G/H$.  
\end{itemize}
\end{proposition}

\begin{proof} It is enough to prove claim (a) in the case that $\phi$ is a character of $G$, represented by a homomorphism $T_\phi : G \to {\rm GL}_d(\bQ^{c})$. In this case, the general result of \cite[Chapter II, Section 4, Corollary]{M} implies that  
\begin{equation*}\label{trace root num eq}
W({\rm Tr}(\phi)) = W(\phi) W(\bar{\phi}) = \mathrm{det}_{\phi}(\rho_F(-1)), 
\end{equation*}
where ${\rm det}_\phi$ is the homomorphism $G^{\rm ab} \to (\bQ^{c})^{\times}$ induced by sending each $g$ in $G$ to ${\rm det}(T_\phi(g))$ and $\rho_F$ is the reciprocity map $F^\times \to G^{\rm ab}$. In addition, $-1$ belongs to $O^\times_{F}$ and so is sent by $\rho_F$ to an element of the inertia subgroup of $G^{\rm ab}$ of order dividing two. In particular, since this inertia group has odd order, one has $\rho_F(-1) = 1$ and so $\mathrm{det}_{\phi}(\rho_F(-1))=1$. This proves claim (a). 

To prove part (b), we use the inductivity of local root numbers in degree zero to compute  
\begin{align*}W({\rm Ind}_H^G\phi) =&\, W({\rm Ind}_{H}^{G}(\phi - \phi(1){\bf 1}_H))W({\rm Ind}_H^G1_H)^{\phi(1)}\\
                                                      =&\, W(\phi - {\bf 1}_H)W({\rm Ind}_{H}^G{\bf 1}_H)^{\phi(1)}\\
                                                      =&\, W(\phi)W({\bf 1}_H)^{-1}\prod_{\theta\in (G/H)^*}W(\theta)^{\phi(1)},\end{align*}
where $(G/H)^*$ denotes the group of homomorphisms $G/H \to (\bQ^{c})^{\times}$, and the last equality is true because ${\rm Ind}_H^G{\bf 1}_H$ is equal to the sum of $\theta$ over $(G/H)^*$. Now, if $G/H$ is odd, respectively even, then the only real-valued functions in $(G/H)^*$ are ${\bf 1}_G$, respectively ${\bf 1}_G$ and $\chi_{G/H}$, and all other homomorphisms occur in complex conjugate pairs. The result of part (b) therefore follows from the above displayed formula after isolating the conjugate pairs in the product that occurs in the final term, applying the result of part (a) to each of these pairs, and noting that $W({\bf 1}_H) = W({\bf 1}_G) = 1$. 

To prove part (c) we recall that, by a result of Fr\"ohlich and Queyrut \cite[Section 4, Theorem 3]{FQ73}, one has  $W(\phi) = \phi(\rho_L(x))$, where $\rho_L$ is the reciprocity map $L^\times\to H$ and $x$ is any element of $L\setminus F$ with $x^2 \in F^\times$. In addition, since $\phi$ is of dihedral-type, it is trivial on restriction to $F^\times$ (cf. \cite[Section 3, Lemma 1]{FQ73}) and so $\phi(\rho_L(x))^2 = \phi(\rho_L(x^2)) = \phi(1) = 1$. On the other hand, the order of $\phi$ is odd (since it divides $|H| = |G|/2$ which, under the given hypothesis on $|G|$, is odd) and so $\phi(\rho_L(x))^2 = 1$ implies $\phi(\rho_L(x)) =1$ and hence also $W(\phi)=1$. 

This last equality then combines with a straightforward application of the general result of part (b) to prove the formula in part (c).  
%
%
\end{proof}

We are now ready to prove Theorem \ref{root number theorem}. At the outset we note that $G$ is the semi-direct product of its inertia subgroup $I$ by the cyclic quotient group $G/I$. We further note that, by assumption, the group $I$ is cyclic of odd order, and hence, in particular, that $G$ is supersolvable.  

Fix $\chi$ in ${\rm Symp}(G)$. Then, since $\chi$ is tamely ramified, one has $W(\chi)\in \{\pm 1\}$ (cf. \cite[Chapter III, Theorem 21(iii)]{Fr83}) and so $W(2\chi) = W(\chi)^2 = 1$. It is therefore enough for us to prove that $W(\psi_2(\chi)) = 1$. 

Next we note that, by Lemma \ref{psi_2 ind - ind psi_2}(c)(i), there exists a subgroup $J$ of $G$ that contains $I$ and a linear character $\phi$ of $J$ such that one has $\chi = {\rm Ind}_{J}^G\phi$. In particular, since $J$ contains $I$ and $G/I$ is cyclic, there exists a normal subgroup $H$ of $G$ with $J \trianglelefteq H \trianglelefteq G$  and such that $H/J$ is cyclic of $2$-power order and $G/H$ is cyclic of odd order. 

Then one has $\chi = {\rm Ind}_H^G\chi'$ with $\chi' := {\rm Ind}_J^H\phi$ and we claim that $\chi'$ belongs to ${\rm Symp}(H)$. To see this we note $\chi'$ is an irreducible character of $H$ (since $\chi$ is irreducible) and so, by the Frobenius-Schur Theorem (cf. \cite[Theorem (73.13)]{CR2}), the sum $c_H(\chi') := |H|^{-1}\sum_{h \in H}\chi(h^2)$ belongs to $\{-1, 0, 1\}$ and is equal to $-1$ if and only if $\chi'$ is symplectic. In addition, since $H$ is normal in $G$ and of odd index one has $g^2 \in H \Longleftrightarrow g \in H$ for each $g \in G$ and so 
\begin{align*}c_G(\chi) = c_G({\rm Ind}_H^G\chi') =&\, |G|^{-1}\sum_{g \in G}({\rm Ind}_H^G\chi')(g^2)\\
 =&\, |G|^{-1}\sum_{\tau \in \mathcal{T}}\sum_{h \in H}(\chi')^\tau(h^2)\\
  =&\, |\mathcal{T}|^{-1}\sum_{\tau\in \mathcal{T}}c_H((\chi')^\tau)\end{align*} 
where $\mathcal{T}$ is a set of coset representatives of $H$ in $G$
and $(\chi')^\tau$ is the irreducible character of $H$ that sends each
element $h$ to $\chi'(h^\tau)$. In particular, since both $c_G(\chi) =
-1$ (as $\chi\in {\rm Symp}(G)$) and each $c_H((\chi')^\tau)$ belongs
to $\{-1,0,1\}$, the displayed equality implies that
$c_H((\chi')^\tau) = -1$ for all $\tau$. Thus one has $c_H(\chi') =
-1$ and so $\chi' \in {\rm Symp}(H)$, as claimed.

Now, since $G/H$ is cyclic of odd order, one has $W(\psi_2(\chi)) =
W({\rm Ind}_H^G(\psi_2(\chi')) = W(\psi_2(\chi'))$, where the first
equality follows from Lemma \ref{psi_2 ind - ind psi_2}(b) and the
second from Proposition \ref{fq result}(b). Thus, if necessary after
replacing $G$ by $H$ (and $\chi$ by $\chi'$), we can assume in the
sequel that $\chi$ has $2$-power degree.

Next we note that, since $G$ is supersolvable, an induction theorem of
Martinet (cf. \cite[Chapter III, Theorem 5.2]{M}) implies that either
$\chi = {\rm Tr}({\rm Ind}_{H'}^G\phi')$, where $\phi'$ is a linear
character of some subgroup $H'$ of $G$, or that $\chi$ is the
induction to $G$ of a quaternion character of a subgroup. In view of
Proposition \ref{fq result}(a), we can therefore also assume in the
sequel that there exists a subgroup $J_1$ of $G$ that has $2$-power
index, and hence contains $I$, and a quaternion character $\phi_1$ of
$J_1$ such that $\chi = {\rm Ind}_{J_1}^G\phi_1$.

This implies $J_1$ has a quotient $Q$ isomorphic to a generalized quaternion group and that 
\be \lab{E:dagger}
\phi_1 = {\rm Inf}^{J_{1}}_{Q}  ({\rm Ind}^{Q}_{P} \theta),
\ee
where $P$ is the cyclic subgroup of $Q$ of index $2$ and $\theta$ a homomorphism $P\to (\bQ^{c})^{\times}$. Let $J_1'$ denote the inverse image of $P$ under the quotient map $J_1 \to Q$, and set $\phi_1':= {\rm Inf}^{J_1'}_{P} \theta$ (so $\phi_1'$ is a linear character of $J_1'$). Then the subgroup $J_1'$ is of index $2$ in $J_1$, and \eqref{E:dagger} implies 
\be\label{index 2 induction}
\phi_1 = {\rm Ind}^{J_{1}}_{J_1'} \phi_1'. 
\ee

Now, as $J_1'$ has $2$-power index in $G$, it contains $I$. Thus, since $G/I$ is cyclic, one has $J_1' \trianglelefteq G$ and $G/J_1'$ is cyclic of $2$-power order. In particular, since the degree $(\psi_2(\phi_1))(1) = \phi_1(1)$ is even, one therefore has  
\[ W(\psi_2(\chi)) = W(\psi_2({\rm Ind}_{J_1}^G\phi_1)) = W({\rm Ind}_{J_1}^G(\psi_2(\phi_1))) = W(\psi_2(\phi_1)),\]
where the second equality follows from Lemma \ref{psi_2 ind - ind psi_2}(c)(iii) (after taking account of (\ref{index 2 induction})) and the third from 
Proposition \ref{fq result}(b).

In addition, since $Q$ is the Galois group of a tamely ramified
extension of $p$-adic fields that has odd ramification degree, it is
the semi-direct product of a cyclic (inertia) subgroup of odd order by
a cyclic group. In particular, since such a group can have no quotient
isomorphic to $H_8$, the group $Q$ must be isomorphic to $H_{4m}$,
with $m$ odd. In view of (\ref{E:dagger}), we can therefore apply
Lemma \ref{psi_2 ind - ind psi_2}(d) (with $\Gamma, \Delta$ and $\phi$
taken to be $Q, P$ and $\theta$) to deduce that
\[ 
W(\psi_2(\phi_1)) = W(\psi_2({\rm Ind}_P^Q\theta)) = W({\rm
  Ind}_{P/N}^{Q/N}(\lambda))W(\chi_{Q/P}),
\]
where $N$ denotes the centre of $Q$ (so $N$ is the unique subgroup of $P$ of order two) and $\lambda$ denotes $\theta^2$, regarded as a homomorphism $P/N \to (\bQ^{c})^{\times}$. 

Finally, since the group $Q/N$ is generalized dihedral with $|Q/N| = 2m \equiv 2$ modulo $4$, and the inertia subgroup of $Q/N$ has odd order, the theorem of Fr\"ohlich and Queyrut implies (via Proposition \ref{fq result}(c)) that 
$W({\rm Ind}_{P/N}^{Q/N}(\lambda)) = W(\chi_{Q/P})$. Upon subsituting this fact into the last two displayed formulas, we deduce that $W(\psi_2(\chi)) = W(\chi_{Q/P})^2 =  1$. 

This completes the proof of Theorem \ref{new addition}.  


\section{Symplectic Galois-Jacobi Sums II} \lab{S:gj2}

We retain the notation of the previous two sections. For any real
number $x$, we write $\sgn(x) \in \{ \pm 1\}$ for the sign of $x$. In
this section we shall examine $\sgn(J^*(L_\pi/L, \chi))$ for $\chi \in
\Symp(G)$. This will in turn lead to the definition of  $\cJ^{\ast}_\infty(F_\pi/F)
\in \Cl(\bZ G)$ for $F$ a number field and $[\pi] \in H^1_t(F,G)$.

Recall that for each $\chi \in R_G$, the adjusted
Galois-Gauss sum is defined (in \cite[Chapter IV, Section 1]{Fr83}) by setting
\[
\tau^*(L, \chi):=\tau(L, \chi)y(L,\chi)^{-1}z(L,\chi), 
\]
for suitable roots of unity $y(L, \chi)$ and $z(L,\chi)$ in
$\bQ^c$. \cite[Chapter IV, Theorem 29(i)]{Fr83} implies that $y(K,\chi) =1$
for all $\chi$ in ${\rm Symp}(G)$. One can also check (directly from
the definitions) that $z(L,\psi_2(\chi)) = z(L,\chi)^2$ and hence that
$z(L,\chi) = z(L,\psi_2(\chi)) = 1$ for each $\chi$ in ${\rm
  Symp}(G)$.

Recall that Theorem \ref{new addition} asserts that $J(L_\pi/L, \chi)
>0$ whenever $\chi \in \Symp(G)$. The following result is now a direct
consequence of the definition of the adjusted Galois-Jacobi sum
$J^*(L_\pi/L, \chi)$.

\begin{theorem} \lab{T:sgn}
Suppose that $\chi \in \Symp(G)$. Then
\[
\sgn(J^*(L_\pi/L, \chi) = \sgn(y(L_\pi/L, \psi_2(\chi))).
\]
\qed
\end{theorem}

The following Propostion shows that $\sgn(y(L_\pi/L, \psi_2(\chi))) = -1$
is possible.

\begin{proposition} \lab{P:neg}
Let $M/L$ be a tamely ramified Galois extension with $\Gamma:=
\Gal(M/L) \simeq H_{4m}$, with $m$ odd. Suppose that the inertia
subgroup $\Gamma_{0}$ of $\Gamma$ is odd. Then for each 
$\chi \in \Symp(G)$, we have $y(M/L, \psi_2(\chi)) = -1$.
\end{proposition}

\begin{proof}
For ease of notation, we write e.g. $y(\chi)$ rather than $y(M/L,\chi)$.
 
To prove the desired result, we shall use Lemma~\ref{psi_2 ind - ind
  psi_2}. Let $\Delta$ be the cyclic subgroup of $\Gamma$ of index
$2$. Then all irreducible symplectic characters of $\Gamma$ can be
written in the form $\chi={\rm Ind}_\Delta^{\Gamma}\phi$, where $\phi$
is a linear character of $\Delta$. It is easy to see that the order of
$\phi$ does not divide $2$ (for otherwise ${\rm Ind}_\Delta^{\Gamma}\phi$
would be an orthogonal character of $\Gamma$; see \cite[Chapter III, 
Theorem 3.1]{M}), and that $\phi$ (and hence also $\phi^2$) is non-trivial on
$\Gamma_{0}$ (since $\Gamma_{0}$ has odd order).

Let $Z$ denote the centre of $\Gamma$ and let $\chi_{\Gamma/\Delta}$ denote 
the unique non-trivial homomorphism $\Gamma/\Delta \rightarrow (\bQ^{c})^{\times}$. 
Using the formula in Lemma~\ref{psi_2 ind - ind psi_2}(d), one can compute that 
\begin{align*}
y(\psi_{2}(\chi)) =&\, y(\psi_2({\rm Ind}_\Delta^{\Gamma} \phi))\\
=&\, y( {\rm Inf}_{\Gamma/Z}^{\Gamma}({\rm Ind}_{\Delta/Z}^{\Gamma/Z}(\phi^2))) 
\cdot y( {\rm Inf}_{\Gamma/\Delta}^{\Gamma}(\chi_{\Gamma/\Delta})) 
\cdot y( {\bf 1}_{\Gamma})^{-1} \\
=&\, (-1)^{\deg(n_0)}{\rm det}_{n_0}(\sigma) 
\cdot (-1)\chi_{\Gamma/\Delta}(\sigma) 
\cdot (-1) {\bf 1}_{\Gamma}(\sigma)^{-1}\\
=&\, 1 \cdot 1 \cdot (-1) = -1,	
\end{align*} 
where $\phi^2$ is regarded as a character of $\Gamma/Z$, $\sigma$ 
is the Frobenius element in $\Gamma/\Gamma_{0}$ lifted to $\Gamma$, 
and $n_0:= n({\rm Inf}_{\Gamma/Z}^{\Gamma}({\rm Ind}_{\Delta/Z}^{\Gamma/Z}(\phi^2)))$ 
denotes the unramified part (cf. \cite[Chapter I, (5.6)]{Fr83}) of 
${\rm Inf}_{\Gamma/Z}^{\Gamma}({\rm Ind}_{\Delta/Z}^{\Gamma/Z}(\phi^2))$. 
The third equality above holds since clearly 
${\rm Inf}_{\Gamma/\Delta}^{\Gamma}(\chi_{\Gamma/\Delta})$ and ${\bf 1}_{\Gamma}$ 
are both linear and unramified. The fourth equality follows from the fact 
that $n_0=0$ (since $\phi^2$ is irreducible and ramified, by 
\cite[Chapter III, Proposition 1.3(ii)]{Fr83} the unramified part 
$n({\rm Ind}_{\Delta/Z}^{\Gamma/Z}(\phi^2))=0$ and therefore $n_0=0$).
\end{proof}

The above discussion motivates the following definition.

\begin{definition}  \lab{D:sjac}
We define $J^{*}_{\infty}(L_\pi/L,-) \in \Hom_{\Omega_{\bQ}}(R_G,
J(\bQ^c))$ by its values on $\chi \in \Irr(G)$ as follows:
\[
J^{*}_{\infty}(L_\pi/L, \chi)_v = 
\begin{cases}
\sgn(J^{*}(L_\pi/L, \chi)) \quad &\text{if $\chi \in \Symp(G)$
 and $v|\infty$;}\\ 
1 \quad &\text{otherwise.} 
\end{cases}
\]

We write $J^{*}_{\infty}(L_\pi/L)$ for the element of $K_{0}(\bZ G,
\bQ)$ represented by the homomorphism $J^{*}_{\infty}(L_\pi/L,-)$.
Similarly, we also write $J^{*}(L_\pi/L)$ for the element of
$K_{0}(\bZ G, \bQ)$ represented by $J^{*}(L_\pi/L,-)$.  
\qed
\end{definition}

\begin{theorem} \lab{T:jdet}
We have
\[
J^{*}(L_{\pi}/L,-) \cdot J^{*}_{\infty}(L_{\pi}/L,-)^{-1} \in \Det(\bQ^c
G),
\]
and so
\[
\partial^0(J^{*}(L_\pi/L)) = \partial^0(J^{*}_{\infty}(L_\pi/L)).
\]
\end{theorem}

\begin{proof}
To ease notation, set $f = J^{*}(L_{\pi}/L,-) \cdot
J^{*}_{\infty}(L_{\pi}/L,-)^{-1} $. 

Then, since $f \in
\Hom_{\Omega_{\bQ}}(R_G, (\bQ^{c})^{\times})$ the
Hasse-Schilling-Maass Norm Theorem (cf. \cite[Theorem (7.48)]{curtisr})
implies that the first equality is equivalent to asserting that
$f(\chi)$ is a strictly positive real number for every $\chi$ in
$\Symp(G)$. This in turn follows at once from the definition of
$J^{*}_{\infty}(L_{\pi}/L,-)$.

The second equality is now an immediate consequence of the fact that
$\partial^0(\Det(\bQ^c G)) = 0$.
\end{proof}

Suppose now that $F$ is a number field, and that $[\pi] \in H^1_t(F,
G)$. We also recall that $F_{\pi, v}:= F_{\pi} \otimes_{F} F_{v}
\simeq F_{v, \pi_v}$ (see e.g. \cite[(2.4)]{M87}).

\begin{definition} \lab{D:jclass}
We set
\[
J^{*}(F_\pi/F):= \sum_{v \nmid \infty} J^{*}(F_{v,\pi_{v}}/F_{v})
\in K_0(\bZ G, \bQ),
\]
and
\[
J^{*}_{\infty}(F_\pi/F):= \sum_{v \nmid \infty} J^{*}_{\infty}(F_{v,\pi_{v}}/F_{v})
\in K_0(\bZ G, \bQ).
\]
(Note that the infinite sums make sense as
  $J^{*}_{\infty}(F_{v,\pi_{v}}/F_{v}) = J^{*}(F_{v,\pi_{v}}/F_{v})
  =0$ for all places $v$ that are unramified in $F_{\pi}/F$.)

We define $\cJ^*(F_{\pi}/F) \in \Cl(\bZ G)$ by
\[
\cJ^*(F_{\pi}/F) := \partial^0(J^{*}(F_\pi/F)), \quad
\cJ^*_\infty(F_{\pi}/F) := \partial^0(J^{*}_{\infty}(F_\pi/F))
\]
(see \ref{E:rkes}).
\qed
\end{definition}

\begin{proposition}  \lab{E:jeq}
Suppose that $F$ is a number field, and $[\pi] \in H^1_t(F,G)$. Then
\[
\cJ^*(F_\pi/F) = \cJ^*_{\infty}(F_\pi/F).
\]
\end{proposition} 

\begin{proof}
This is a direct consequence of Theorem \ref{T:jdet} and Definition \ref{D:jclass}.
\end{proof}


\section{Proof of Theorem \ref{T:tboas}} \lab{S:fp}
Let $[\pi] \in H^1_t(F,G)$, and write
\[
\fc(\pi) = [A_{\pi}, O_FG;\br_G] - [O_{\pi}, O_FG; \br_G] \in
K_0(O_FG,F) \subseteq K_0(O_FG, F^c).
\]
For each finite place $v$ of $F$, we write $[\pi_v]$ for the image of $[\pi]$
in $H^{1}_{t}(F_v, G)$.

Recall that 
\[
K_0(O_FG, F) \simeq \frac{\Hom_{\Omega_F}(R_G, J_{f}(F^c))}{\prod_{v
    \nmid \infty} \Det(O_{F_v}G)^{\times}}.
\]
A representing homomorphism in $\Hom_{\Omega_F}(R_G, J_{f}(F^c))$ of
$\fc(\pi)$ is $f = (f_{v})_{v}$ defined by
\[
f_{v}(\chi) = \vpi_{v}^{\langle \psi_{2}(\chi) - 2\chi, s_{v} \rangle_{G}},
\]
using the notation of Corollary \ref{C:rephom}. Let $\Ram(\pi)$ denote
the set of finite places of $F$ at which $F_{\pi}/F$ is ramified. If $v
\notin \Ram(\pi)$, then $s_v =1$ and so $f_v =1$.

\begin{definition} \lab{D:ramdec}
Suppose that $v \in \Ram(\pi)$. Then we define $\fc(\pi;v) \in
K_0(O_FG, F)$ to be the element represented by $f^{(v)} =
(f^{(v)}_{w})_{w} \in \Hom_{\Omega_F}(R_G, J_{f}(F^c))$ given by
\[
f^{(v)}_{w}(\chi) =
\begin{cases}
f_v(\chi) = \vpi_{v}^{\langle \psi_{2}(\chi) - 2\chi, s_{v} \rangle_{G}}
&\text{if $w=v$}; \\
1 &\text{otherwise.}
\end{cases}
\]
\end{definition}

\begin{lemma} \lab{L:ramdec}
We have
\be \lab{E:ramdec}
\fc(\pi) = \sum_{v \in \Ram(\pi)} \fc(\pi; v).
\ee
\end{lemma}

\begin{proof}
It follows from the definitions that
\[
f = \prod_{v \in \Ram(\pi)} f^{(v)},
\]
and this implies the result.
\end{proof}


We can now prove Theorem \ref{T:tboas}.

\begin{theorem} \lab{T:main}
Suppose that $[\pi] \in H^{1}_{t}(F, G)$ and that $A_{\pi}$ is
defined. Then
\[
\partial^{0}(\cN_{F/\bQ}(\fc(\pi)) \cdot \cJ^{*}_{\infty}(F_\pi/F)^{-1} = 0, 
\]
and so there is an equality
\[
(A_{\pi}) - (O_{\pi}) = \cJ^{*}_{\infty}(F_\pi/F),
\]
i.e. (see \eqref{E:mjt})

\[
(A_{\pi}) - W(F_\pi/F) = \cJ^{*}_{\infty}(F_\pi/F),
\]
in $\Cl(\bZ G)$.
\end{theorem}

\begin{proof}
Lemma \ref{L:ramdec} implies that in order to show that
\[
\partial^{0}(\cN_{F/\bQ}(\fc(\pi)) \cdot \cJ^{*}_{\infty}(F_\pi/F)^{-1}
  = 0,
\]
 it suffices to show that 
\[
\partial^{0}(\cN_{F/\bQ}(\fc(\pi;v))
  \cdot \cJ^{*}_{\infty}(F_{v,\pi_{v}}/F_v)^{-1}= 0
\]
for each $v \in \Ram(\pi)$. Theorem \ref{T:jdet} implies that this
is equivalent to showing that 
\[
\partial^{0}(\cN_{F/\bQ}(\fc(\pi;v))
\cdot \cJ^*(F_{v,\pi_{v}}/F_v)^{-1}= 0
\]
for each $v \in \Ram(\pi)$.

We see from the description of ${\rm Cl}(\bZ G)$ given in Theorem
\ref{T:kdes}(a) that this last equality will in turn follow if, for each $v \in
\Ram(\pi)$, we show that
\[
J^{*}(F_{v,\pi_{v}}/F_{v},-)^{-1}  \cdot
(\cN_{F/\bQ}(f^{(v)})) \in \prod_{l} \Det(\bZ_{l}G)^{\times}.
\]

To show this last inclusion, we first observe that Proposition~\ref{P:jkill}(a)
implies that the inclusion holds at all rational primes $l$ not lying
below $v$.

For each rational prime $l$ that lies below $v$, we fix an embedding
$\Loc_l: \bQ^c \to \bQ^c_{l}$ and use it to identify $\Irr(\Gamma)$
with $\Irr_l(\Gamma)$. We recall in particular that such an
isomorphism $R_G \to R_{G, l}$ in turn induces an isomorphism
$\Hom_{\Omega_{F}}(R_G, (\bQ^c)^{\times}_{l}) \to
\Hom_{\Omega_{F_{v}}}(R_{G, l}, (\bQ^c_l)^{\times})$ 
(cf. \cite[Chapter II, Lemma 2.1]{Fr83}).  
Then, reasoning analogously to the proof of
\cite[Theorem 19, pp. 114--116]{Fr83}, 
 one can deduce from Proposition \ref{P:jkill}(b) that
\[ 
\cN_{F_{v}/\bQ_{l}}(f_{v}) \cdot  \Loc_l \bigl(\cN_{F/\bQ}(f^{(v)}) \bigr)^{-1} \in \Det(\bZ_{l}G).
\]
This establishes the desired inclusion at rational primes lying below
$v$ and completes the proof of the desired result.

		

		
\end{proof}

\begin{remark}
Let us make some remarks concerning Theorem \ref{T:main} when
$F_{\pi}/F$ is locally abelian.

Suppose that $v \in \Ram(\pi)$. Set $s_v:= \pi(\sigma_v)$, and write
$H_v := \langle s_v \rangle$. Proposition \ref{P:gd}(d) with $G = H_v$
and Proposition \ref{P:adams}(b) imply that for each $\chi \in
R_{H_v}$, we have
\begin{align*}
\langle \chi, s_v \rangle^{*}_{H_v} - \langle \chi, s_v \rangle_{H_v} &=
(d(s_v), \chi)_{H_v} \\
&= \langle \psi_2(\chi) - \chi, s_v \rangle_{H_v}.
\end{align*}

Now suppose also that $F_v$ contains a primitive $|s_v|$-th root of
unity. This implies in particular that the extension
  $F^{\pi_{v}}_{v}/F_v$ is abelian. Let $\fb(\pi;v) \in K_0(FH_{v},
F)$ be the element represented by $\rho^{(v)} = (\rho^{(v)}_{w})_{w}
\in \Hom_{\Omega_{F}}(R_{H_v}, J_{f}(F^c))$ defined by
\[
\rho^{(v)}_{w}(\chi) =
\begin{cases}
\vpi_{v}^{(d(s_{v}), \chi)_{H_{v}}} =
\vpi_{v}^{\langle \psi_{2}(\chi) - \chi, s_{v} \rangle_{H_v}}
&\text{if $w=v$}; \\
1 &\text{otherwise}
\end{cases}
\]
Observe that without the hypothesis concerning the number of roots of
unity in $F_v$, we would only have that $\rho^{(v)} \in \Hom(R_{H_v},
J_{f}(F^c))$ rather than $\rho^{(v)} \in \Hom_{\Omega_{F}}(R_{H_v},
J_{f}(F^c))$. We also see from the definitions of $\fc(\pi;v)$ and
$\fb(\pi;v)$ (see also \eqref{E:indeq1} and \eqref{E:indeq3}) that
$\fc(\pi;v) = \Ind^{G}_{H_v} \fb(\pi;v)$.

Hence if for every $v \in \Ram(\pi)$, $F_v$ contains
a primitive $|s_v|$-th root of unity--which is precisely what happens if
$F_{\pi}/F$ is locally abelian--then we have
\be \lab{E:locab}
\fc(\pi) = \sum_{v \in \Ram(\pi)} \Ind^{G}_{H_v} \fb(\pi;v),
\ee
and so (using \eqref{E:indcd})
\begin{align*}
\partial^0(\fc(\pi)) &= \sum_{v \in \Ram(\pi)}
\partial^0(\Ind^{G}_{H_v} \fb(\pi;v)) \\
&= \sum_{v \in \Ram(\pi)}
\Ind^{G}_{H_v} \partial^0(\fb(\pi;v)) \\
&=0.
\end{align*}
We now deduce from Theorem \ref{T:main} that $\cJ^*_{\infty}(F_\pi/F)
= 0$.

A comparison of \eqref{E:locab} and \eqref{E:ramdec} highlights the
crucial difference between the locally abelian case and the general
case. In both cases, the class $\fc(\pi)$ may be decomposed into a sum
over the places $v \in \Ram(\pi)$ of classes $\fc(\pi;v) \in K_0(O_FG,
F^c)$. However, in the locally abelian case, these classes
$\fc(\pi;v)$ are induced from cyclic subgroups of $G$, while in the
general case they are not. This is why Theorem \ref{T:main} may be
proved in the locally abelian case using abelian Jacobi sums thereby
showing that in this situation $\cJ^*_\infty(F_\pi/F) = 0$), which is
what is done in \cite{CV16}.  \qed
\end{remark}


\section{Proof of Theorem \ref{T:ce}} \lab{S:ce}


Let $F$ be any imaginary quadratic field such that $\Cl(O_F)$ contains
an element of order $4$. In this section we shall construct infinitely
many counterexamples to Conjecture \ref{C:boas} by showing that if
$\ell$ is any sufficiently large prime with $\ell \equiv 3 \pmod{4}$
and $G$ is the generalised quaternion group $H_{4\ell}$, then there
are infinitely many tame $G$-extensions $F_{\pi}/F$ of fields such
that $A_\pi$ exists and $\cJ^*_\infty(F_\pi/F) \neq 0$. Hence, for
these extensions, $(O_\pi) \neq (A_\pi)$ in $\Cl(\bZ G)$. This will
prove Theorem \ref{T:ce}.


In what follows we fix an imaginary quadratic field $F$ such that
$\Cl(O_F)$ contains an element of order $4$. To prove
Theorem~\ref{T:ce}, it will suffice to prove the following result,
which we shall derive as a consequence of works of Fr\"ohlich (see
\cite{Fr74}).

\begin{lemma}\lab{l:extn}
Suppose that $\ell$ is a sufficiently large prime
and that $G\simeq H_{4\ell}$. Then, there exists a $G$-extension
$F_\pi/F$ of fields such that:
\begin{itemize}
\item[(a)] $F_\pi/F$ is ramified at only a single prime $\fp$ of $F$
  with $\fp\nmid \ell$;
\item[(b)] The prime $\fp$ does not split in $F_{\pi}/F$;
\item[(c)] The ramification index of $\fp$ is equal to $\ell$;
\end{itemize}
\end{lemma}

Before we prove this result, we shall first show that Lemma~\ref{l:extn} 
implies Theorem~\ref{T:ce}. \\

\begin{proof}[{Proof of Theorem~\ref{T:ce}}]
First we note that the decomposition subgroup of $G$ at $\fp$ is equal to
$H_{4\ell}$.  We also recall that, for an odd prime $\ell$, the
generalised quaternion group $H_{4\ell}$ has a single, irreducible,
non-trivial symplectic character $\chi$, say.  

If $\fq$ is unramified in $F_\pi/F$, then one has $\sgn(y(F_{\pi,
  \fq}/F_{\fq},\psi_{2}(\chi)))=1$. On the other hand,
Theorem~\ref{T:sgn} and Proposition~\ref{P:neg} imply that
\[ 
\sgn(J^{\ast}(F_{\pi, \fp}/F_{\fp}, \chi)) 
= \sgn(y(F_{\pi, \fp}/F_{\fp},\psi_{2}(\chi))) = -1.  
\]
In particular, if we now assume in addition that $\ell \equiv 3 \pmod{4}$, then it follows
from \cite[Chapter II, Proposition 4.4]{Fr83} that the element
$\cJ^*_\infty(F_\pi/F) \in \Cl(\bZ G)$ (see Definition~\ref{D:sjac}
and \ref{D:jclass}, and Proposition~\ref{E:jeq}) is non-trivial. (We
remark in passing that if instead $\ell \equiv 1 \pmod{4}$, then the
same argument shows that $\cJ^*_\infty(F_\pi/F) = 0$.)
\end{proof}

The remainder of this section will be devoted to the construction of
the extensions described in Lemma \ref{l:extn}.



Let $L$ be an unramified,  cyclic extension of $F$ of degree $4$.  We
write $E/F$ for the quadratic subextension of $L/F$ and write
$\varphi_{E/F}$ for the quadratic character of $E/F$ on ideals
of $F$.  We also view this as an idele class character of $F$. If
$\omega$ denotes the idele class character of $E$ that cuts out the
extension $L/E$, then $\omega$ is of quaternion type (i.e. the
restriction of $\omega$ to $J(F)$ is equal to
$\varphi_{E/F}$---see \cite[p. 405]{Fr74}.)

For each prime $\ell$, the symbol $\eta_\ell$ will denote a primitive
$\ell$-th root of unity.  Then, following \cite[Theorem 4]{Fr74}, we
consider the following conditions on primes.

\begin{condition}\label{cd:ell}
Let $\ell$ be an odd prime such that:
\begin{itemize}
\item[(a)]	$[F(\eta_\ell) : F]$ is even;
\item[(b)] $E\not\subseteq F(\eta_\ell +\eta_\ell^{-1})$;
\item[(c)] the class number of $E$ is not divisible by $\ell$.
\end{itemize}
\end{condition}
We remark that these properties are satisfied for all sufficiently
large $\ell$. (We observe, in particular, that in our case
\ref{cd:ell}(b) is automatically satisfied for sufficiently large
$\ell$ since $E/F$ is unramified.)

Henceforth we therefore fix a prime $\ell$ satisfying \ref{cd:ell} and
abbreviate $\eta_\ell$ to $\eta$.  We then write $\Sigma_{-}$ for the
set of primes $\fp$ of $F$ satisfying the following properties (see
\cite[(8.5)]{Fr74}).

\begin{condition}\label{cd:prime}
Let $\fp$ be a finite prime of $F$ such that:
\begin{itemize}
\item[(a)] The prime $\fp$ is inert in $E/F$ (ie. $\varphi_{E/F}(\fp) = -1$);
\item[(b)] $N_{F/\bQ} \equiv -1 {\pmod \ell}$.
\end{itemize}
\end{condition}

In what follows, if $\fp \in \Sigma_{-}$, we write $\fp_{E}$ for the
unique prime of $E$ lying above $\fp$.

Our argument relies on the following result of Fr\"ohlich (see
\cite[pp. 432--434]{Fr74}). We state the result, and then describe an
outline of the proof. We refer the reader to \cite{Fr74} for complete
details. 

\begin{theorem}\label{T:ali}
There are infinitely many primes in $\Sigma_{-}$ (in fact a
subset of positive Chebotarev density) for which the following
statement is true: there exists a non-trivial idele class character $\theta$ of $E$ of
order $\ell$, and of dihedral type (i.e. the restriction of $\theta$
to $J(F)$ is trivial) which is ramified at $\fp_{E}$ and which is
unramified at all other finite places of $E$.
\end{theorem} 

\begin{proof}
We remark that necessary conditions for such a $\theta$ to exist are
given in \cite[Section 8, Lemma 5]{Fr74}. The existence of
$\theta$ is demonstrated on pp. 433--434 of loc.cit. via the following
argument. 

Recall that $\eta$ is a primitive $\ell$-th root of unity, and set
\[
M:= E(\eta).
\]
(Note that this field is denoted by $L$ in \cite[p. 433, l. 9]{Fr74},
which is an unfortunate clash of notation with the field $L$ defined
earlier in loc. cit. (see \cite[p. 407]{Fr74}).

Write $\wt{M}$ for the extension of $M$ obtained by adjoining the
elements 
\[
\{ y^{1/\ell} \mid y \in O^{\times}_{E} \}.
\]
It is shown in loc. cit. that, for each prime $\fp$ of $F$ satisfying
the following Frobenius conditons, there exists an idele class
character $\theta$ of $E$ satisfying the properties we seek:

\begin{condition}\label{cd:frob}
For every prime $\mathfrak{P}$ of $\widetilde{M}$ lying above $\fp$,
the Frobenius element $\delta=(\mathfrak{P}, \widetilde{M}/F)$
satisfies:
\begin{itemize}
\item[(F1)] $\delta^2=1$;
\item[(F2)] $\delta |_{E}$ is non-trivial 
(so $\fp$ does not split in $E/F$);
\item[(F3)] $\delta \mid_{F(\eta)}$ is non-trivial 
(so $\fp$ satisfies Property~\ref{cd:prime}(b) above).
\end{itemize}
\end{condition}

The set of primes $\fp$ of $F$ satisfying Property~\ref{cd:frob} 
has positive Chebotarev density, and all such primes lie in $\Sigma_{-}$.
\end{proof}

Let $\theta$ be an idele class character of $E$ as constructed in
Theorem \ref{T:ali}, and let $N/E$ denote the extension cut out by
$\theta$. Then $N/E$ is cyclic of order $\ell$, ramified (necessarily
totally) at $\fp_E$, and at no other primes of $E$. As $\theta$ is of
dihedral type, the extension $N/F$ is dihedral of order $2\ell$.

Set $\psi:= \omega \theta$. Then $\psi$ is an idele class character of
$E$ of quaternion type, and we deduce that $F_{\pi(\psi)}:= NL$ is an
$H_{4\ell}$ extension of $F$. (Note that the field that we call
$F_{\pi(\psi)}$ is denoted by the symbol $F_\psi$ in \cite{Fr74}.) The
extension $F_{\pi(\psi)}/F$ is ramified only at $\fp$, with
ramification index $\ell$. We have the following diagram of fields and
corresponding idele class characters (where we write $\vphi$ for
$\vphi_{E/F}$):

\begin{center}
\begin{tikzpicture}[node distance = 1.5cm, auto]
\node (F) {$F$};
\node (E) [above of=F] {$E$};
\node (L) [above of=E, right of=E] {$L$};
\node (N) [above of=E, left of=E] {$N$};
\node (NL) [above of=E, node distance = 3cm] {$F_{\pi(\psi)}=NL$};
			
\node (M) [right of=E, node distance = 2cm] {$M=E(\eta)$};
\node (M1) [right of=L, node distance = 2cm] {$\widetilde{M}$};
			
\draw[-] (F) to node [swap]{$\varphi$}(E);
\draw[-] (E) to node {$\omega$} (L);
\draw[-] (E) to node [swap]{$\theta$} (N);
\draw[-] (E) to node {}(M);
\draw[-] (F) to node {}(N);
\draw[-] (L) to node {} (NL);
\draw[-] (N) to node {}(NL);
\draw[-] (F) to node {}(M);
\draw[-] (M) to node {}(M1);
\draw[-] (E) to node {}(M1);
\end{tikzpicture}
\end{center}

To complete the proof of Lemma \ref{l:extn}, it suffices to show that,
in Theorem \ref{T:ali}, there are infinitely many choices of $\fp$
(and so of $\theta$) such that the decomposition group of $\fp$ in
$F_{\pi(\psi)}/F$ is not abelian. This is equivalent to imposing an
additional Frobenius condition on $\fp$. In order to do this, we
require the following lemma.

\begin{lemma}\label{l:extn lemma}
The extension $\widetilde{M}/E$ and $L/E$ are linearly disjoint. Hence
$[\wt{M}L:\wt{M}] = 2$.
\end{lemma}

\begin{proof}
The extension $\widetilde{M}/E$ has a unique quadratic sub-extension,
viz. the unique quadratic sub-extension of $M/E$ (recall that $M =
E(\eta)$). This extension is ramified at places above $\fp$, and so
cannot be equal to the unramified quadratic extension $L/E$.
\end{proof}

We now fix an element $\delta_1 \in \Gal(\wt{M}L/F)$ which maps under
the obvious quotient map onto the element $\delta \in \Gal(\wt{M}/F)$
constructed in the proof of Theorem \ref{T:ali} (see \eqref{cd:frob}),
and we consider the set of primes $\fp$ of $F$ satisfying the
following Frobenius condition:

\begin{condition}\label{cd:frob4}
For every prime $\mathfrak{Q}$ of $\widetilde{M}L$ lying
above $\fp$, 
\begin{itemize}
\item[(F4)] the Frobenius element $(\mathfrak{Q},
  \widetilde{M}L/F)$ lies in the conjugacy class of $\delta_1$.
\end{itemize}
\end{condition}

The set of primes $\fp$ satisfying \eqref{cd:frob4} has positive
Chebotarev density, and plainly if $\fp$ satisfies \eqref{cd:frob4},
then it also satisfies \eqref{cd:frob}. 

Suppose that $\fp$ satisfies \eqref{cd:frob4}. Then the corresponding
extension $F_{\pi(\psi)}/F$ constructed above is an
$H_{4\ell}$-extension unramified outside $\fp$, in which $\fp$ is
non-split and ramified, with ramification index $\ell$. Hence
$F_{\pi(\psi)}/F$ an extension satistisfying the conditions of
Lemma~\ref{l:extn}.

This completes the proof of Lemma \ref{l:extn}.

\begin{remark}
It is shown in \cite[Theorem 4]{Fr74} that for the extensions
$F_{\pi(\psi)}/F$ constructed above satisfying the conditions of Lemma
\ref{l:extn}, we have
\[
W(F_{\pi(\psi)}/F) = \varphi_{E/F}(\fp) = -1.
\]
This implies that $(O_{\pi(\psi)})\neq 0$ (see \eqref{E:mjt}), and
  so, since $\cJ^*_\infty(F_{\pi(\psi)}/F) \neq 0$, it follows from
  Theorem \ref{T:tboas} that $(A_{\pi(\psi)}) = 0$. \qed
\end{remark}

\begin{remark} \lab{R:db}
Dominik Bullach has explained to us how explicit counterexamples to
Conjecture \ref{C:boas} can also be derived from Theorem \ref{T:tboas}
by using general results of Neukirch on the embedding problem (see
\cite{N73}) rather than the explicit computations of Fr\"ohlich in
\cite{Fr74}.  
\qed
\end{remark}

\end{document}